\documentclass[a4paper,reqno]{amsart}

\usepackage[british]{babel}

\numberwithin{equation}{section}

\usepackage[left=3.5cm,right=3.5cm,top=3.5cm,bottom=3.5cm,footskip=1.5cm,headsep=1cm]{geometry}
\usepackage[%
pdfauthor={Habib Ammari, Silvio Barandun, Jinghao Cao, and Florian Feppon},%
pdftitle={Edge modes in subwavelength resonators in one dimension},
pdfsubject={Edge modes in subwavelength resonators in one dimension},
pdfkeywords={Subwavelength resonances, one-dimensional periodic chains of subwavelength
resonators, non-Hermitian topological systems, topologically protected edge modes},
]{hyperref}

\usepackage{lmodern} 
\usepackage{lipsum}
\usepackage{mathtools}
\usepackage{amsmath}
\usepackage{amssymb}
\usepackage{tikz}
\usetikzlibrary{matrix,positioning,decorations.pathreplacing,calc}
\usetikzlibrary{
decorations.text,%
decorations.markings,%
shadows}
\usepackage[export]{adjustbox}
\usepackage{graphicx} 

\usepackage{caption}
\usepackage{subcaption}

\usepackage{dsfont} 
\usepackage[format=hang]{caption}

\usepackage[normalem]{ulem}

\newcommand{\abs}[1]{\lvert#1\rvert}
\usepackage{bm}

\usepackage[
style=numeric,
sorting=nty,
maxnames=99,
natbib=true,
sortcites]{biblatex}

\DeclareNameAlias{default}{family-given}

\AtEveryBibitem{
 \clearfield{url}
 \clearfield{issn}
 \clearfield{isbn}
 \clearfield{urlyear}
 
 \ifentrytype{book}{
  \clearfield{pages}}{%
 }
}

\usepackage{xargs}                      
\usepackage[prependcaption,textsize=tiny,textwidth=3cm]{todonotes}
\newcommandx{\unsure}[2][1=]{\todo[linecolor=red,backgroundcolor=red!25,bordercolor=red,#1]{#2}}
\newcommandx{\change}[2][1=]{\todo[linecolor=blue,backgroundcolor=blue!25,bordercolor=blue,#1]{#2}}
\newcommandx{\info}[2][1=]{\todo[linecolor=OliveGreen,backgroundcolor=OliveGreen!25,bordercolor=OliveGreen,#1]{#2}}
\newcommandx{\improvement}[2][1=]{\todo[linecolor=black,backgroundcolor=black!25,bordercolor=black,#1]{#2}}
\newcommandx{\thiswillnotshow}[2][1=]{\todo[disable,#1]{#2}}
\usepackage{amsthm}
\usepackage[noabbrev,capitalize]{cleveref}
\crefname{prop}{Proposition}{Propositions}
\crefname{equation}{}{}

\newtheorem{theorem}{Theorem}[section]
\newtheorem{lemma}[theorem]{Lemma}
\newtheorem{proposition}[theorem]{Proposition}

\theoremstyle{definition}
\newtheorem{definition}[theorem]{Definition}

\newtheorem{corollary}[theorem]{Corollary}
\newtheorem{remark}[theorem]{Remark}

\crefname{ass}{Assumption}{Assumptions}
\crefname{dfn}{Definition}{Definitions}
\crefname{cor}{Corollary}{Corollaries}
\crefname{enumi}{Point}{Points}
\usepackage{fancyhdr}

\fancypagestyle{plain}{%
\fancyhead[C]{}
\fancyfoot[C]{}

}
\fancypagestyle{nosection}{%
\fancyhead[CE]{}
\fancyhead[CO]{}
\fancyfoot[CE,CO]{\thepage}
}
\DeclareMathOperator{\N}{\mathbb{N}}
\DeclareMathOperator{\Z}{\mathbb{Z}}

\DeclareMathOperator{\R}{\mathbb{R}}
\DeclareMathOperator{\C}{\mathbb{C}}

\renewcommand{\i}{\mathbf{i}}
\renewcommand{\tilde}{\widetilde}
\renewcommand{\hat}{\widehat}
\renewcommand{\bar}[1]{\overline{#1}}
\newcommand{\inv}{^{-1}}

\renewcommand{\qedsymbol}{$\blacksquare$}

\DeclareMathOperator{\diag}{diag}
\DeclareMathOperator{\BO}{\mathcal{O}}
\DeclareMathOperator{\exactC}{\mathbf{C}^\alpha}
\DeclareMathOperator{\capmat}{\mathcal{C}^\alpha}
\newcommand{\capmati}[1]{\mathcal{C}^\alpha_{#1}}
\DeclareMathOperator{\gencapmat}{\mathcal{C}^\alpha_G}

\newcommand{\pri}{^\prime}
\renewcommand{\epsilon}{\varepsilon}
\newcommand{\dd}{\ensuremath \,\mathrm{d}}

\newcommand{\zak}[1]{\varphi^\text{zak}_{#1}}
\newcommand{\per}{{\rm per}}

\DeclareMathOperator{\crystal}{\mathfrak{C}}

\DeclareMathOperator{\iL}{{\mathsf{L}}}
\DeclareMathOperator{\iR}{{\mathsf{R}}}
\DeclareMathOperator{\iLR}{{\mathsf{L},\mathsf{R}}}

\makeatletter
\def\grd@save@target#1{%
  \def\grd@target{#1}}
\def\grd@save@start#1{%
  \def\grd@start{#1}}
\tikzset{
  grid with coordinates/.style={
    to path={%
      \pgfextra{%
        \edef\grd@@target{(\tikztotarget)}%
        \tikz@scan@one@point\grd@save@target\grd@@target\relax
        \edef\grd@@start{(\tikztostart)}%
        \tikz@scan@one@point\grd@save@start\grd@@start\relax
        \draw[minor help lines] (\tikztostart) grid (\tikztotarget);
        \draw[major help lines] (\tikztostart) grid (\tikztotarget);
        \grd@start
        \pgfmathsetmacro{\grd@xa}{\the\pgf@x/1cm}
        \pgfmathsetmacro{\grd@ya}{\the\pgf@y/1cm}
        \grd@target
        \pgfmathsetmacro{\grd@xb}{\the\pgf@x/1cm}
        \pgfmathsetmacro{\grd@yb}{\the\pgf@y/1cm}
        \pgfmathsetmacro{\grd@xc}{\grd@xa + \pgfkeysvalueof{/tikz/grid with coordinates/major step}}
        \pgfmathsetmacro{\grd@yc}{\grd@ya + \pgfkeysvalueof{/tikz/grid with coordinates/major step}}
        \foreach \x in {\grd@xa,\grd@xc,...,\grd@xb}
        \node[anchor=north] at (\x,\grd@ya) {\pgfmathprintnumber{\x}};
        \foreach \y in {\grd@ya,\grd@yc,...,\grd@yb}
        \node[anchor=east] at (\grd@xa,\y) {\pgfmathprintnumber{\y}};
      }
    }
  },
  minor help lines/.style={
    help lines,
    step=\pgfkeysvalueof{/tikz/grid with coordinates/minor step}
  },
  major help lines/.style={
    help lines,
    line width=\pgfkeysvalueof{/tikz/grid with coordinates/major line width},
    step=\pgfkeysvalueof{/tikz/grid with coordinates/major step}
  },
  grid with coordinates/.cd,
  minor step/.initial=.2,
  major step/.initial=1,
  major line width/.initial=2pt,
}

\tikzset{
    style1/.style={
      matrix of math nodes,
      every node/.append style={text width=#1,align=center,minimum height=5ex},
      nodes in empty cells,
      left delimiter=(,
      right delimiter=),
      },
    style2/.style={
      matrix of math nodes,
      every node/.append style={text width=#1,align=center,minimum height=5ex},
      nodes in empty cells,
      left delimiter=\lbrace,
      right delimiter=\rbrace,
      }
}
\makeatother

\newdimen\XCoord
\newdimen\YCoord
\newcommand*{\ExtractCoordinate}[1]{\path[reset cm] (#1);
\pgfgetlastxy{\XCoord}{\YCoord}; 
\pgfmathsetmacro{\XCoord}{\XCoord*1pt/1cm};
\pgfmathsetmacro{\YCoord}{\YCoord*1pt/1cm};}

\begin{document}

\title{Edge modes in subwavelength resonators in one dimension}

\author[H. Ammari]{Habib Ammari}
\address{\parbox{\linewidth}{Habib Ammari\\
ETH Zurich, Department of Mathematics, Rämistrasse 101, 8092 Zurich, Switzerland}}
\author[S. Barandun]{Silvio Barandun}
\address{\parbox{\linewidth}{Silvio Barandun\\
ETH Zurich, Department of Mathematics, Rämistrasse 101, 8092 Zurich, Switzerland}}
\email{{\includegraphics[height=5pt]{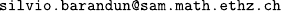}}}

\author[J. Cao]{Jinghao Cao}
\address{\parbox{\linewidth}{Jinghao Cao\\
ETH Zurich, Department of Mathematics, Rämistrasse 101, 8092 Zurich, Switzerland}}

\author[F. Feppon]{Florian Feppon}
\address{\parbox{\linewidth}{Florian Feppon\\
NUMA research Unit, KU Leuven, 3001 Leuven, Belgium}}

\begin{abstract}
    We present the mathematical theory of one-dimensional infinitely periodic chains of subwavelength resonators. We analyse both Hermitian and non-Hermitian systems. Subwavelength resonances and associated modes can be accurately predicted by a finite dimensional eigenvalue problem involving a capacitance matrix. We are able to compute the Hermitian and non-Hermitian Zak phases, showing that the former is quantised and the latter is not. Furthermore, we show the existence of localised edge modes arising from defects in the periodicity in both the Hermitian and non-Hermitian cases. In the non-Hermitian case, we provide a complete characterisation of the edge modes.
\end{abstract}

\date{}

\maketitle

\medskip

\par

\vfill
\hrule
\tableofcontents
\vspace{-0.5cm}
\hrule
\vfill
\newpage

\section{Introduction}
In the past decade, controlling and manipulating waves via interaction with objects at subwavelength scales has gained a lot of attention in both photonics and phononics \cite{lemoult.fink.ea2011,lemoult.kaina.ea2016,yves.fleury.ea2017}. One way to achieve subwavelength interactions is to use high-contrast metamaterials, that are media constituted by the insertion of a set of highly contrasted resonators into a background medium. Here, subwavelength means that the size of such resonators is much smaller than the operating wavelength. A typical example in acoustics of such high-contrast resonators are air bubbles in water, which give rise to Minneart resonances \cite{ammari.fitzpatrick.ea2018}. Examples in electromagnetics include high-contrast dielectric particles and plasmonic particles \cite{ammari.li.ea2023Mathematical,ammari.millien.ea2017Mathematical}.

High-contrast subwavelength resonators have been extensively studied in the three-di\-men\-sion\-al case \cite{ammari.davies2020,  ammari.davies.ea2022a, ammari.fitzpatrick.ea2018, devaud.hocquet.ea2008, feppon.ammari2022}. Recently, an increased interest has been dedicated to topological properties of one-dimensional resonators. Studies include one-dimensional infinite periodic media with continuous material parameters \cite{lin.zhang2022}, simplified Su-Schrieffer-Heeger (SSH) models \cite{craster.davies2022} and various physical experiments \cite{thomas.hughes2004, wang2010}. A rigorous mathematical analysis of the finite one-dimensional case was recently presented \cite{feppon.cheng.ea2022}. The present work completes this analysis by considering the one-dimensional periodic case. Since the interactions between the subwavelength resonators only imply the nearest neighbors in one-dimension, it also connects the field of high-contrast metamaterials to condensed-matter physics.  

In classical wave systems, sources of amplification and dissipation of energy can be modelled by non-real material parameters making the underlying system non-Hermitian, meaning that the left and right eigenmodes are distinct. Our work considers both Hermitian and non-Hermitian systems of subwavelength resonators. We look at these two cases separately as they present deep underlying differences. A particular case of non-Hermitian systems are those with parity-time (PT-) symmetry. Recently, non-Hermitian subwavelength resonators have been studied in three dimensions \cite{ammari.hiltunen2020}, but to the best of our knowledge no literature exists on the one-dimensional case, which is of interest not only for the study of one-dimensional metamaterials but also of quantum systems as the interactions in both cases are short-range.

In this work, we are able to show that, similarly to the three-dimensional case \cite{ammari.davies.ea2020}, also in the one-dimensional case it is possible to design subwavelength structures where certain frequencies cannot propagate and are trapped near an edge. This typically happens by introducing a defect in the geometry --- in the Hermitian case --- or in the material parameters --- in the non-Hermitian case. Generally these localised modes are sensitive with respect to small perturbations. In order to manufacture structures presenting the said characteristics, stability with respect to perturbations is required. We take inspiration from quantum mechanics where so-called topological insulators have been extensively studied \cite{drouot2021, drouot.fefferman.ea2020a,fefferman.lee-thorp.ea2018, fefferman.lee-thorp.ea2014a}. The underlying principle of these structures is the existence of a topological invariant that captures the propagation properties of the system. In the present setup, the correct topological invariant is the \emph{Zak phase}. The combination of two structures having different invariants will give rise to modes that are confined at the interface of the structure and that are stable with respect to imperfections. These modes are known as topologically protected edge modes. We first compute the Zak phase for both the Hermitian and non-Hermitian cases and prove that the non-Hermitian one is not quantised. Then we show the existence of the said edge modes and demonstrate their robustness. Moreover, in the non-Hermitian case, we also provide a full characterisation of these modes.

The paper is organized as follows. In \Cref{sec: Problem statement}, we present the mathematical setup of the problem. In \Cref{sec: quasiperiodic Dirichlet to Neumann}, we introduce the Dirichlet-to-Neumann map and solve the exterior problem. \Cref{sec: subwavelength resonances} is dedicated to deriving an asymptotic approximation of the subwavelength resonances and their associated modes. We show that a generalised eigenvalue problem involving the capacitance matrix solves this. The exact tridiagonal structure of the capacitance matrix allows us to study the topological properties of one-dimensional systems of subwavelength resonators without dilute regime assumptions. In \Cref{sec: Hermitian}, we focus on the Hermitian case and show numerically the existence of edge modes in the presence of geometrical defects. Ultimately in \Cref{sec: non-Hermitian}, we first explicitly compute the Zak phase and then prove the existence of an edge mode in the case of defects in the periodicity of the material parameters. The robustness of the edge modes in both the Hermitian and non-Hermitian cases is illustrated numerically.

\section{Problem statement and preliminaries}\label{sec: Problem statement}
\subsection{Problem formulation}
We consider a one-dimensional system constituted of $N$ periodically repeated disjoint subwavelength resonators $D_i\coloneqq (x_i^{\iL},x_i^{\iR})$, where $(x_i^{\iLR})_{1\leq i\leq N} \subset \R$ are the $2N$ extremities satisfying $x_i^{\iL} < x_i^{\iR} <  x_{i+1}^{\iL}$ for any $0\leq i\leq N-1$. We assume without loss of generality that $x_1^{\iL}=0$.
We denote by $(x_i^{\iLR})_{i\in\N}$ the infinite sequence obtained by setting
\begin{align*}
x_{i+N}^{\iLR}:=x_i^{\iLR}+L,
\end{align*}
for some $L>x_N^{\iR}-x_1^{\iL}$.
Furthermore,  we let $D^n = \bigcup_{i=1}^{N}D_i + nL$ so that $D^n = D + nL$ is just the repetition of $D^0\eqqcolon D$. We denote the entire structure by $\crystal\coloneqq \bigcup_{n\in\Z}D^n$.
We also denote by  $\ell_i = x_i^{\iR} - x_i^{\iL}$ the length of the $i$-th resonators,  and by $s_i= x_{i+1}^{\iL} -x_i^{\iR}$ the spacing between the $i$-th and $(i+1)$-th resonator.
With our convention, the spacing $s_{N}$ is the distance that separates the last resonator of a unit cell from the first of the next one:
\begin{align*}
  s_{N}:=x_{N+1}^{\iL}-x_N^{\iR}=L-x_N^{\iR}+x_1^{\iL}.
\end{align*}
One notices that $L = \sum_{i=1}^{N} \ell_i + s_{i}$ is the size of the unit cell, which we denote by $Y\coloneqq (0,L)$. The system is illustrated on \Cref{fig:setting}.

    \begin{figure}[h]
        \centering
        \begin{adjustbox}{width=\textwidth}
        \begin{tikzpicture}
            \pgfmathsetseed{3}
            \coordinate (0) at (0,0);
            \foreach \i in {1,2,...,7}{
                \pgfmathsetmacro{\r}{max(1.5*rnd,0.7)}
                \pgfmathsetmacro{\im}{\i-1}
                \coordinate (\i) at ($(\im)+(\r,0)$);
           };
            \foreach \i in {0,2,...,5}{
                \pgfmathsetmacro{\im}{\i}
                \pgfmathsetmacro{\ip}{\i+1}
                \pgfmathsetmacro{\lin}{int(\i/2+1)}
                \draw[|-|] (\im) -- (\ip) ; 
                \draw ($0.5*(\im)+0.5*(\ip)$) node[above] {$\ell_{\lin}$};
                \draw node[below,yshift=-0.2] at (\im) {\small $x_{\lin}^{\iL}$}; 
                \draw node[below,yshift=-0.2] at (\ip) {\small $x_{\lin}^{\iR}$}; 
            };
            \foreach \i in {1,3,...,4}{
                \pgfmathsetmacro{\im}{\i}
                \pgfmathsetmacro{\ip}{\i+1}
                \pgfmathsetmacro{\lin}{int((\i-1)/2+1)}
                \pgfmathsetmacro{\linp}{int((\i-1)/2+2)}
                \draw[dotted] (\im) -- ($(\im)+(0,0.7)$);
                \draw[dotted] (\ip) -- ($(\ip)+(0,0.7)$);
                \draw[|-|,dashed] ($(\im)+(0,0.6)$) -- ($(\ip)+(0,0.6)$);
                \draw node[above] at ($0.5*(\im)+0.5*(\ip)+(0,0.6)$) {$s_{\lin}$};
            }       
            \coordinate (8) at ($(5)+(2,0)$);
            \coordinate (9) at ($(5)+(3,0)$);
            \coordinate (C) at ($0.5*(5)+0.5*(8)$);
            \draw[|-|] (8) -- (9) ; 
            \draw node[below,yshift=-0.2] at (8) {\small $x_{N-1}^{\iL}$}; 
            \draw node[below,yshift=-0.2] at (9) {\small $x_{N-1}^{\iR}$}; 
            \draw ($0.5*(8)+0.5*(9)$) node[above] {$\ell_{N-1}$};
            \coordinate (10) at ($(9)+(1,0)$);
            \coordinate (11) at ($(9)+(2.3,0)$);
            \draw[|-|] (10) -- (11) ; 
            \draw node[below,yshift=-0.2] at (10) {\small $x_{N}^{\iL}$}; 
            \draw node[below,yshift=-0.2] at (11) {\small $x_{N}^{\iR}$}; 
            \draw ($0.5*(10)+0.5*(11)$) node[above] {$\ell_{N}$};
            \draw[dotted] (9) -- ($(9)+(0,0.7)$);
            \draw[dotted] (10) -- ($(10)+(0,0.7)$);
            \draw[|-|,dashed] ($(9)+(0,0.6)$) -- ($(10)+(0,0.6)$);
            \draw node[above] at ($0.5*(9)+0.5*(10)+(0,0.6)$) {$s_{N-1}$};
            \foreach \i in {-2,-1,0,1,2} {
                \pgfmathsetmacro\c{0.03}
                \coordinate (a) at ($(C)+(\i*1/3,0)$);
                \fill ($(a)-(\c,\c)$)  rectangle ($(a)+(\c,\c)$);};
            \draw[color=black,draw opacity=0.1] ($(0)-(1.5,0)$) -- (5);
            \draw[color=black,draw opacity=0.1,dashed] ($(0)-(1.5,0)$) --      ($(0)-(2.2,0)$);
            \draw[color=black,draw opacity=0.1] (8) -- ($(11)+(1.5,0)$);

            \coordinate (12) at ($(11)+(1.5,0)$);
            \ExtractCoordinate{12}
            \draw[color=black,draw opacity=0.1] ($(11)+(1.5,0)$) -- ($(11)+(3,0)$);
            \draw[|-] ($(12)+(1,0)$) -- ($(12)+(1.5,0)$);
            \draw[-,dashed] ($(12)+(1.5,0)$) -- ($(12)+(2,0)$);
            \draw node[below] at ($(12)+(1,0)$) {$x_{N+1}^{\iL}=x_1^{\iL}+L$};
            \draw[dotted] (11) -- ($(11)+(0,0.7)$);
            \draw[dotted] ($(12)+(1,0)$) -- ($(12)+(1,0)+(0,0.7)$);
            \draw[|-|,dashed] ($(11)+(0,0.6)$) -- ($(12)+(1,0)+(0,0.6)$);
                \draw node[above] at ($0.5*(11)+0.5*(12)+(0.5,0.6)$) {$s_{N}$};

            \tikzset{pattern/.pic={
                \pgfmathsetseed{3}
                \coordinate (0) at (0,0);
                \fill (-1,0) circle  (2pt);
                \foreach \i in {1,2,...,5}{
                    \pgfmathsetmacro{\r}{max(1.5*rnd,0.7)}
                    \pgfmathsetmacro{\im}{\i-1}
                    \coordinate (\i) at ($(\im)+(\r,0)$);
               };
                \foreach \i in {0,2,...,5}{
                    \pgfmathsetmacro{\im}{\i}
                    \pgfmathsetmacro{\ip}{\i+1}
                    \pgfmathsetmacro{\lin}{int(\i/2+1)}
                    \draw[|-|] (\im) -- (\ip) ; 
                };
                \coordinate (8) at ($(5)+(2,0)$);
                \coordinate (9) at ($(5)+(3,0)$);
                \coordinate (C) at ($0.5*(5)+0.5*(8)$);
                \draw[|-|] (8) -- (9) ; 
                \coordinate (10) at ($(9)+(1,0)$);
                \coordinate (11) at ($(9)+(2.3,0)$);
                \draw[|-|] (10) -- (11) ; 
                \foreach \i in {-2,-1,0,1,2} {
                    \pgfmathsetmacro\c{0.03}
                    \coordinate (a) at ($(C)+(\i*1/3,0)$);
                    \fill ($(a)-(\c,\c)$)  rectangle ($(a)+(\c,\c)$);};
                \fill ($(11)+(1.5,0)$) circle  (2pt);
                \draw[color=black,draw opacity=0.1] ($(0)-(1,0)$) -- (5);
                \draw[color=black,draw opacity=0.1] (8) -- ($(11)+(1.5,0)$);
            }
            }
            \foreach \i in {0,1,...,5}{
                \pic[scale=0.2] (pattern) at ($(-1.1+\i*\XCoord*0.2+\i*0.2,-2)$) {pattern};
                \pgfmathsetmacro{\p}{int(\i-2)}
                \draw node[below] at ($(-1.1+\i*\XCoord*0.2+\i*0.2,-2)$) {$\p L$};
            };
            \draw[dotted] (-2,-2)-- (-1.3,-2);
            \draw[dotted] ($(-1.1+6*\XCoord*0.2+5*0.2,-2)$)-- (14,-2);
            \draw[postaction={decorate,decoration={markings,mark=at position 0.5 with
            {\arrow{>}}}},dotted] ($(-1.1+2*\XCoord*0.2+0.4,-2)$) to [controls=+(90:1) and +(-70:1)] (0,0);
            \draw[postaction={decorate,decoration={markings,mark=at position 0.5 with
            {\arrow{>}}}},dotted] ($(-1.1+3*\XCoord*0.2+2*0.2+0.2,-2)$) to      [controls=+(90:1) and +(-30:1)] ($(12)+(1,0)$);
        \end{tikzpicture}
        \end{adjustbox}
        \caption{An infinite chain  of $N$ subwavelength resonators, with lengths
        $(\ell_i)_{1<i\leq N}$ and spacings $(s_{i})_{1\leq i\leq N-1}$,
        periodically repeated with period $L$.}
        \label{fig:setting}
    \end{figure}
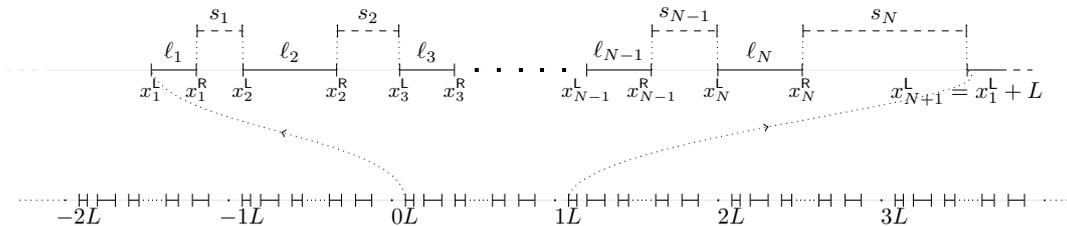
As a wave field $u(t,x)$ propagates in a heterogeneous medium, it is solution to the following one-dimensional wave equation:
\begin{equation}
\label{eqn:6gct3}
    \frac{1}{\kappa(x)}\frac{{\partial}^{2}}{{\partial} t^{2}}u(t,x) -\frac{{\partial}}{{\partial} x}\left(
    \frac{1}{\rho(x)}\frac{{\partial}}{{\partial} x}  u(t,x)\right) = 0, \qquad (t,x)\in\R\times\R.
\end{equation}
The parameters $\kappa(x)$ and $\rho(x)$ are the material parameters of the medium. We consider 
\begin{equation}
\label{eqn:jk5r8}
    \kappa(x)=
    \begin{dcases}
        \kappa_i & x\in D_i+L\Z,\\
        \kappa&  x\in\R\setminus \crystal,
    \end{dcases},\quad
    \rho(x)=
    \begin{dcases}
        \rho_b & x\in {\crystal},\\
        \rho&  x\in\R\setminus \crystal,
    \end{dcases}
\end{equation}
where $\rho_b, \rho \in \R_{>0}$ {and $i\in\Z$}. We are interested in both the Hermitian $\kappa_i\in\R_{>0}$ and the non-Hermitian $\kappa_i\in\C$ (with nonzero imaginary parts) cases. In the Hermitian case we typically set for simplicity $\kappa_i=\kappa_b$ for some $\kappa_b\in\R$ for all $1\leq i\leq N$. We stick with the general notation allowing different  $\kappa_i$'s, but think of them as equal to a positive constant in the Hermitian case.

Following the notation of \cite{ammari.fitzpatrick.ea2018, feppon.ammari2022}, the wave speeds inside the resonators $\crystal$ and inside the background medium $\R\setminus \crystal$, are denoted respectively by $v_i$ and $v$, the wave numbers respectively by $k_i$ and $k$, and the contrast between the $\rho$'s of the resonators and the background medium by $\delta$:
\begin{align}
    v_i:=\sqrt{\frac{\kappa_i}{\rho_b}}, \qquad v:=\sqrt{\frac{\kappa}{\rho}},\qquad
    k_i:=\frac{\omega}{v_i},\qquad k:=\frac{\omega}{v},\qquad
    \delta:=\frac{\rho_b}{\rho}.
\end{align}
 
Up to using a  Fourier decomposition in time, we can assume that 
the total wave field $u(t,x)$ is time-harmonic:
\begin{equation}
\label{eqn:t3qz2}
    u(t,x)=\Re ( e^{-\i\omega t}u(x) ),
\end{equation}
for a function $u(x)$ which solves the one-dimensional Helmholtz equations:
\begin{equation}
\label{eq:time indip Helmholtz}
    -\frac{\omega^{2}}{\kappa(x)}u(x)-\frac{\dd}{\dd x}\left( \frac{1}{\rho(x)}\frac{\dd}{\dd
    x}  u(x)\right) =0,\qquad x \in\R.
\end{equation}

In these circumstances of step-wise defined material parameters, the wave problem determined by \eqref{eq:time indip Helmholtz} can be rewritten as the following system of coupled one-dimensional Helmholtz equations:
\begin{align}
    \label{eq: system of coupled equations}
    \begin{dcases}
        \frac{\dd{^2}}{\dd x^2}u(x)+ \frac{\omega^2}{v^2}u(x) = 0, & x\in \R \setminus \crystal ,\\
        \frac{\dd{^2}}{\dd x^2}u(x)+ \frac{\omega^2}{v_i^2}u(x) = 0, & x\in D_i+L\Z ,\\
        u\vert_{\iR}(x^{\iLR}_{{i}}) - u\vert_{\iL}(x^{\iLR}_{{i}}) = 0, &  \forall {{i}}\in \Z ,\\
        \left.\frac{\dd u}{\dd x}\right\vert_{\iR}(x^{\iL}_{{i}}) - \delta\left.\frac{\dd u}{\dd x}\right\vert_{\iL}(x^{\iL}_{{i}}) = 0, & \forall {{i}}\in \Z ,\\
        \delta\left.\frac{\dd u}{\dd x}\right\vert_{\iR}(x^{\iR}_{{i}}) - \left.\frac{\dd u}{\dd x}\right\vert_{\iL}(x^{\iR}_{{i}}) = 0, & \forall {{i}}\in \Z ,
    \end{dcases}
\end{align}
{for $i\in\Z$ and }where for a one-dimensional function $w$ we denote by
\begin{align*}
    w\vert_{\iL}(x) \coloneqq \lim_{\substack{s\to 0\\ s>0}}w(x-s) \quad \mbox{and} \quad  w\vert_{\iR}(x) \coloneqq \lim_{\substack{s\to 0\\ s>0}}w(x+s)
\end{align*}
if the limits exist. 
\subsection{Floquet-Bloch theory}
To study this periodic problem we use Floquet-Bloch theory (see, for instance, \cite{ammari.fitzpatrick.ea2018a,kuchment1993Floquet}).
\begin{definition}
    Given $f(x)\in L^2(\R)$, the Floquet transform of $f$ with period $L$ is defined as
    \begin{align*}
        \mathcal{F}[f](x,\alpha)\coloneqq 
        \sum_{n\in\Z}f(x-{{n}}L)e^{\i \alpha {{n}} L}. 
    \end{align*}
\end{definition}
The Floquet transform is an analogue of the Fourier transform in the periodic case. Also for the Floquet transform, the original function may be recovered from the collection of the transformed ones via the following Plancherel type inversion:
\begin{align*}
    \mathcal{F}\inv[g](x) = \frac{L}{2\pi}\int_{-\frac{\pi}{L}}^{\frac{\pi}{L}} g(x,\alpha)\dd \alpha.
\end{align*}

A function $f:\R\to\C$ is said to be $\alpha$-quasiperiodic if $e^{-\i\alpha x}f(x)$ is periodic. One remarks that $\mathcal{F}[f](x,\alpha)$ is $\alpha$-quasiperiodic in $x$ (with period $L$) and periodic in $\alpha$ (with period $\frac{2\pi}{L}$). We will thus be interested in quasiperiodicities laying in the \emph{first Brillouin zone} $Y^* \coloneqq \R/\frac{2\pi}{L}\Z=(-\frac{\pi}{L},\frac{\pi}{L}]$.

We will denote ${\hat{u}}^\alpha(x)\coloneqq \mathcal{F}[u](x,\alpha){e^{-\i \alpha x}}$ the Floquet transform of a solution to \eqref{eq: system of coupled equations}.
Inserting  ${\hat{u}}^\alpha(x)$ into \eqref{eq:time indip Helmholtz} and using the periodicity of the material parameters
$\kappa(x)$ and $\rho(x)$, we find that ${\hat{u}}^\alpha(x)$ solves 
\begin{align}
\label{eq: ualpha solves}
    -\frac{\omega^{2}}{\kappa(x)}{\hat{u}}^\alpha(x) -\left( \frac{\dd}{\dd x}+\i \alpha
    \right)\left[\frac{1}{\rho(x)}\left( \frac{\dd}{\dd x}+\i \alpha
    \right) {\hat{u}}^\alpha(x)\right]=0,\quad x\in\R,
\end{align}
where $x\mapsto {\hat{u}}^\alpha(x)$ is $L$-periodic.

The following lemma describes the subwavelength resonances --- that is $\omega^\alpha$ for which \eqref{eq: ualpha solves} has a nontrivial solution --- in the Hermitian case.
\begin{lemma}
    Let $\kappa\in\R_{>0}$ and $\kappa_i=\kappa_b \in \R_{>0}$ for all $i$. Then there exists a family of real non-negative eigenfrequencies $(\omega_p^\alpha)_{p\in\N}$ such
    that, for any $p\in\N$:
    \begin{enumerate}
        \item[(i)] $\alpha\mapsto \omega_p^\alpha$ is an analytic, $2\pi/L$-periodic function of $\alpha$ for any $p \geq 1$;
        \item[(ii)] $\alpha\mapsto\omega_0^\alpha$ is an analytic $2\pi/L$--periodic function except as $\alpha\in \frac{2\pi}{L}\Z$, where it has a linear behaviour:
            \[
                \omega_0^\alpha\sim  c |\alpha|+\BO(\alpha^{2}) \text{ as
                }|\alpha|\rightarrow 0,
            \] 
            corresponding to the crossing of the branches $\omega_0^\alpha$ and
            $\omega_0^{-\alpha}$. Furthermore, a direct computation shows that \begin{equation}
            \label{eqn:a4txl}
                c=\sqrt{ \frac{\int_{0}^{L}\frac{1}{\rho(x)}\dd
                x}{\int_{0}^{L}\frac{1}{\kappa(x)}\dd x}} \sim v_b \text{ as
                }\delta\rightarrow 0;
            \end{equation}
        \item[(iii)] For any $\alpha\in (-\pi/L,\pi/L)$, there exists a nontrivial $L$-periodic function $u_p^\alpha(x)$ solution to \eqref{eq: ualpha solves} with $\omega=\omega_p^\alpha$. The function $u_p^\alpha(x)$ is called a \emph{Bloch mode} and can be chosen analytic with respect to the parameter $\alpha \in\R$;
        \item[(iv)] By convention, one can choose 
            \[
                0=\omega_0^{\alpha=0}<\omega_1^{\alpha=0}\leq  \omega_2^{\alpha=0}\leq  \dots, 
            \] 
            where $\omega_p^{\alpha=0}$ is the $p$-th eigenvalue of the symmetric eigenvalue problem
            \eqref{eq: ualpha solves} at $\alpha=0$;
        \item[(v)] $\omega_p^\alpha=0$ with $\alpha\in
            \left(-\frac{\pi}{L},\frac{\pi}{L}\right)$ if and only if $\alpha=0$ and $p=0$, which is associated to the constant Bloch mode. As a consequence, 
            \[
                \omega_p^\alpha>0 \text{ for any }p \geq 1 \text{ or for }p=0 \text{
                    with } \alpha\neq 0;
            \] 
        \item[(vi)] $\omega_p^\alpha=\omega_p^{-\alpha}$ and $\bar{u_p^\alpha(x)}$ is
            a Bloch mode for the quasiperiodicity $-\alpha$.
    \end{enumerate}
\end{lemma}
\begin{proof}
    All these properties result from the fact that 
    \[
        \alpha \mapsto -\left( \frac{\dd}{\dd x}+\i \alpha \right)\left[
            \frac{1}{\rho(x)}\left( \frac{\dd}{\dd x}+\i \alpha \right) \right]
    \] 
    form a holomorphic family of Hermitian operators on the space of
    $H^{1}_{\per}((0,L))$, where  $H^{1}_{\per}((0,L))$ is the usual  Sobolev space of periodic complex-valued functions on $(0,L)$.
     Rellich's theorem ensures in particular that $\alpha\mapsto (\omega_p^\alpha)^{2}$ is analytic, and hence $\omega_p^\alpha$ is analytic for all values of $\alpha$ except maybe $\omega_0^\alpha$ at $\alpha=0$. However, the parity property implies that $(\omega_0^\alpha)^{2}=\BO(\alpha^{2})$ as $\alpha\rightarrow 0$, and then $\omega_0^\alpha/\vert\alpha\vert$ is analytic in $\alpha$.

    The value of $c$ in \eqref{eqn:a4txl} can be found as follows. Denoting
    $\lambda_0(\alpha):=(\omega_0^\alpha)^{2}$, we find by differentiating
    \eqref{eq: ualpha solves} with respect to $\alpha$ that
    \begin{multline}
        \label{eqn:0nlp8}
        \left[-\frac{\lambda_0(\alpha)}{\kappa(x)}-\left( \frac{\dd}{\dd x}+\i \alpha \right)\left[ \frac{1}{\rho(x)}
        \left(
        \frac{\dd}{\dd x}+\i\alpha \right) 
        \right]\right]\frac{\dd}{\dd \alpha} {\hat{u}}_0^\alpha(x) \\=
        \frac{\lambda_0\pri(\alpha)}{\kappa(x)}{\hat{u}}_0^\alpha(x)+\i \frac{1}{\rho(x)}
        \left( \frac{\dd}{\dd x}+\i \alpha \right){\hat{u}}_0^\alpha(x)
        +\left(\frac{\dd}{\dd x}+\i \alpha\right)\left[\frac{1}{\rho(x)}\i\alpha {\hat{u}}_0^\alpha(x)\right].
    \end{multline} 
    Setting $\alpha=0$ and using that $\lambda_0(0)=0$ and ${\hat{u}}_0^\alpha(x)\equiv
    u_0$ is a constant, we obtain that $\lambda_0\pri(0)=0$, and then $\frac{\dd}{\dd
    \alpha}u_0^0(x)=0$. Then, differentiating \eqref{eqn:0nlp8} with respect to $\alpha$ and setting $\alpha=0$, we obtain
    \[
        \left[-\left( \frac{\dd}{\dd x} \right)\left[ \frac{1}{\rho(x)}
        \left(
        \frac{\dd}{\dd x} \right) 
        \right]\right]\frac{\dd^{2}}{\dd \alpha^{2}}u_0^0(x)=
        \frac{\lambda_0^{\prime\prime}(0)}{\kappa(x)}u_0 -2\frac{1}{\rho(x)}u_0. 
    \] 
    From  Fredholm's alternative, this equation admits a $L$-periodic solution if
    and only if 
    \[
        \int_0^{L}\left( \frac{\lambda_0^{\prime\prime}(0)}{\kappa(x)}
        u_0-2\frac{1}{\rho(x)}u_0 \right)\dd x=0,
    \] 
    which yields
\[
    \frac{\lambda_0^{\prime\prime}(0)}{2}=\frac{\int_0^{L}\frac{1}{\rho(x)}\dd
    x}{\int_0^{L}\frac{1}{\kappa(x)}\dd x},
\] 
    and hence \eqref{eqn:a4txl} holds. The asymptotic expansion \eqref{eqn:a4txl} is obtained
    then from the formula
    \[
   \frac{\int_0^{L}\frac{1}{\rho(x)}\dd
    x}{\int_0^{L}\frac{1}{\kappa(x)}\dd x} 
    =\frac{\frac{1}{\rho_b}\sum_{i=1}^{N}\ell_i +\frac{1}{\rho}
    \left(L-\sum_{i=1}^{N}\ell_i\right)}{\frac{1}{\kappa_b}\sum_{i=1}^{N}\ell_i
    +\frac{1}{\kappa} \left(L-\sum_{i=1}^{N}\ell_i\right)}
    =\frac{\sum_{i=1}^{N}\ell_i+\delta \left( L-\sum_{i=1}^{N}\ell_i \right)
    }{\frac{1}{v_b^{2}}\sum_{i=1}^{N}\ell_i+\frac{\delta}{v^{2}}\left(
    L-\sum_{i=1}^{N}\ell_i \right) }=v_b^{2}+O(\delta).
    \] 
\end{proof}

We recall from \cite{ammari.fitzpatrick.ea2018a}  that the 
\emph{subwavelength spectrum} of the operator associated to \eqref{eq:time indip Helmholtz} is given by 
\begin{align*}
    \sigma =\bigcup_{p=0}^{N-1}\bigcup_{\alpha\in Y^*}\omega_p^\alpha,
\end{align*}
both in the Hermitian and the non-Hermitian cases.

This describes the band structure of the subwavelength spectrum of \eqref{eq:time indip Helmholtz}: for each $p$ the spectrum traces out bands $\omega_p^\alpha$ as $\alpha$ varies. In the Hermitian case, the spectrum is said to have a \emph{subwavelength band gap} if, for some $0\leq p\leq N-1$, $\max_\alpha \omega_p^\alpha<\min_\alpha \omega_{p+1}^\alpha$. In the non-Hermitian case, a subwavelength band gap is a connected component of $\C\setminus\sigma$. A band is said to be \emph{non-degenerate} if it does not intersect any other band.

Consequently, we study the equivalent one-dimensional spectral problem in the unit cell $Y$ for the
function $u (x,\alpha):={\hat{u}}^\alpha(x)e^{\i \alpha x}{=\mathcal{F}[u](x,\alpha)}$ for $\alpha \in 
Y^*$:
\begin{align}
    \label{eq: system of coupled equations quasiperiodic}
    \begin{dcases}
        \frac{\dd{^2}}{\dd x^2}u(x)+ \frac{\omega^2}{v^2}u(x,\alpha) = 0, & x\in \R \setminus \crystal ,\\
        \frac{\dd{^2}}{\dd x^2}u(x)+ \frac{\omega^2}{v_i^2}u(x,\alpha) = 0, & x\in D_i+L\Z ,\\
        u\vert_{\iR}(x^{\iLR}_n,\alpha) - u\vert_{\iL}(x^{\iLR}_n,\alpha) = 0, &  \forall n\in \Z,\\
        \left.\frac{\dd u}{\dd x}\right\vert_{\iR}(x^{\iL}_n,\alpha) - \delta\left.\frac{\dd u}{\dd x}\right\vert_{\iL}(x^{\iL}_n,\alpha) = 0, & \forall n\in \Z ,\\
        \delta\left.\frac{\dd u}{\dd x}\right\vert_{\iR}(x^{\iR}_n,\alpha) - \left.\frac{\dd u}{\dd x}\right\vert_{\iL}(x^{\iR}_n,\alpha) = 0, & \forall n\in \Z ,\\
        u(x+L,\alpha) = u(x,\alpha)e^{\i\alpha L} & \text{for almost every } x\in\R ,
    \end{dcases}
\end{align}
and consider the subwavelength resonances for the scattering problem \eqref{eq: system of coupled equations quasiperiodic} by performing an asymptotic analysis in the low-frequency and high-contrast regimes
\begin{equation}
\label{eqn:5qmqp}
\omega\rightarrow 0 \quad \mbox{as } \quad \delta\rightarrow 0.
\end{equation}
For this, we adapt the Dirichlet-to-Neumann approach of \cite{feppon.ammari2022b, feppon.cheng.ea2022} to the one-dimensional quasiperiodic problem \eqref{eq: system of coupled equations quasiperiodic}.

\section{Quasiperiodic Dirichlet-to-Neumann map}\label{sec: quasiperiodic Dirichlet to Neumann}
In this section, we characterize the Dirichlet-to-Neumann map of the Helmholtz operator on the domain $Y$ with the quasiperiodic boundary conditions. We give a fully explicit expression of this operator in \Cref{prop:DTN}, before computing its leading-order asymptotic expansion in terms of $\delta$ in \Cref{prop:expansion DTN}.

\medskip

In all what follows, we denote by $H^1(D)$ the usual Sobolev space of complex-valued functions on $D$ and let $H^{1}_{\per}(\R)$ be the usual Sobolev space of periodic complex-valued functions on $\R$
and $H^{1}_{\per, \alpha}(\R):=\{ u: e^{-\i \alpha x} u \in H^{1}_{\per}(\R)\}$.

Throughout the paper, we also denote by $\C^{2N,\alpha}$ the set of quasiperiodic boundary data
$ f\equiv (f_i^{\iLR})_{i\in\Z}$ satisfying 
\[
f_{i+N}^{\iLR}=e^{\i \alpha L} f_i^{\iLR},
\] 
where 
$f_i^{\iL}$ (respectively $f_i^{\iR}$) refers to the component associated to $x_i^{\iL}$ (respectively to $x_i^{\iR}$) {for $i\in\Z$}.
The space of such quasiperiodic sequences is clearly of dimension $2N$. The following lemma provides an explicit expression for the solution to exterior problems on $\R \setminus \crystal$.

\begin{lemma}
    \label{def:DTN}
    Assume that $k$ is not of the form $k=n\pi/s_{i}$ for some nonzero integer
    $n\in\Z\backslash\{0\}$ and index $1\leq i\leq N$. Then, for any quasiperiodic sequence $(f_i^{\iLR})_{1\leq i\leq N}\in\C^{2N,\alpha}$, there exists a unique solution  $w_f^{\alpha}\in
    H^{1}_{\per,\alpha}(\R)$ to the exterior problem:
\begin{align}
   \label{eqn:defDTN}
   \begin{dcases}
    \left(\frac{\dd^2}{\dd x^2} +k^{2} \right) w_f^{\alpha}(x)
    =0, & x\in \R \setminus \crystal,\\
     w_f^{\alpha}(x_i^{\iLR}) = f_i^{\iLR}, & \forall \; 1\leq i \leq N,\\
     w_f^{\alpha}(x+L)= e^{\i \alpha L}w_f^{\alpha}(x), &  \qquad x\in \R \setminus \crystal.
\end{dcases}
\end{align}
Furthermore, when $k\neq 0$, the solution $w_f^{\alpha}$ reads explicitly 
\begin{equation}
\label{eqn:3dqrq}
        w_f^{\alpha}(x) =   a_i e^{\i k x}+b_i e^{-\i kx}  \text{ if }x\in
            (x_i^{\iR},x_{i+1}^{\iL}),\qquad \forall i\in\Z,
\end{equation}
    where $a_i$ and $b_i$ are given by the matrix-vector product
   \begin{equation}
   \label{eqn:pxggr}
        \begin{pmatrix}
           a_i\\
           b_i
        \end{pmatrix} = -\frac{1}{2 \i \sin(k s_{i})} \begin{pmatrix}
            e^{-\i k x_{i+1}^{\iL}} & -e^{-\i k x_i^{\iR}} \\
            -e^{\i k x_{i+1}^{\iL}} & e^{\i k x_i^{\iR}} 
        \end{pmatrix}
        \begin{pmatrix}
           f_i^{\iR}\\
           f_{i+1}^{\iL}
        \end{pmatrix}.
   \end{equation}
\end{lemma}
\begin{proof}
Identical to \cite[Lemma 2.1]{feppon.cheng.ea2022}.
\end{proof}
\begin{definition}[Dirichlet-to-Neumann map]
  For any $k\in\C$ which is not of the form 
    $n\pi/s_{i}$ for some $n\in\Z\backslash\{0\}$ and $1\leq i\leq N-1$, 
    the Dirichlet-to-Neumann map with wave number $k$ is the linear operator 
      $\mathcal{T}^{k,\alpha}\,:\,\C^{2N}\to \C^{2N}$ defined by 
      \begin{equation}
      \label{eqn:rm93y}
    \mathcal{T}^{k,\alpha}[(f_i^{\iLR})_{1 \leq i \leq N}]=
    \left(\pm\frac{\dd w_f^{\alpha}}{\dd x}(x_i^{\iLR})\right)_{1 \leq i \leq N},
      \end{equation}
where $w_f^{\alpha}$ is  the unique solution to  \eqref{eqn:defDTN}.
\end{definition}

The condition that $k\in\C$ is not of the form $n\pi/s_{i}$ for some $n\in\Z\backslash\{ 0 \}$ and $i\in\Z$ is equivalent to state that $k^{2}$ is not a (quasiperiodic) Dirichlet eigenvalue of $-\dd^{2}/\dd x^{2}$ on $\R\setminus \crystal$. We consider a minus sign in \eqref{eqn:rm93y} on the abscissa $x_i^{\iL}$ because $\mathcal{T}^{k,\alpha}[(f_j^{\iLR})_{1\leq j\leq N}]_i^{\iLR}$ is the normal derivative of $w_f^{\alpha}$ at $x_i^{\iLR}$, with the normal pointing \emph{outward} the segment $(x_i^{\iL},x_i^{\iR})$. This convention allows us to maintain some analogy with the analysis in the three-dimensional setting considered in \cite[Section 3]{feppon.ammari2022b}.
      
In the next proposition, we compute $\mathcal{T}^{k,\alpha}$ explicitly.
\begin{proposition} \label{prop:DTN}
    The Dirichlet-to-Neumann map
    $\mathcal{T}^{k,\alpha}$ admits the following explicit matrix representation: for any $k\in\C\backslash\{ n\pi/s_{i}\,|\, n\in\Z\backslash\{0\},\, 1\leq i\leq N-1 \}$, $ f\equiv (f_i^{\iLR})_{1\leq i\leq N}$, $\mathcal{T}^{k,\alpha}[f]\equiv (\mathcal{T}^{k,\alpha}[ f]_{i}^{\iLR})_{1\leq i\leq N}$ is given by
    \begin{equation}
    \label{eqn:DTNexplicit}
    \begin{pmatrix}
        \mathcal{T}^{k,\alpha}[f]_1^{\iL} \\
        \mathcal{T}^{k,\alpha}[f]_1^{\iR} \\
       \vdots\\
        \mathcal{T}^{k,\alpha}[f]_N^{\iL} \\
        \mathcal{T}^{k,\alpha}[f]_N^{\iR} 
    \end{pmatrix} = T^{k,\alpha}  \begin{pmatrix}
        f_1^{\iL}\\
       f_1^{\iR}\\
       \vdots\\
       f_N^{\iL}\\
       f_N^{\iR}
    \end{pmatrix},
\end{equation}
with
\begin{align}
    \label{eq: Matrix form DTN}
    T^{k,\alpha} = \begin{pmatrix}
        -\frac{k\cos(ks_{N})}{\sin(ks_{N})} &&&&&
        \frac{k}{\sin(ks_{N})}e^{-\i \alpha L} \\
        & A^{k}(s_{1}) & & & & \\
        & & A^{k}(s_{2}) & & & \\
        &  & &\ddots & & \\
        & & & &  A^{k}(s_{(N-1)}) &\\
        \frac{k}{\sin(ks_{N})}e^{\i \alpha L}  & & & & & -\frac{k\cos(ks_{N})}{\sin(ks_{N})} \\
    \end{pmatrix},
\end{align} where for any real $\ell\in\R$, $A^{k}(\ell)$ denotes the $2\times 2$ symmetric matrix given by
    \begin{equation}
    \label{eqn:1lzi8}
        A^{k}(\ell):=\begin{pmatrix}
            -\dfrac{k \cos(k\ell)}{\sin(k\ell)} & \dfrac{k}{\sin(k\ell)} \\
            \dfrac{k}{\sin(k\ell)} & -\dfrac{k\cos(k\ell)}{\sin(k\ell)}
        \end{pmatrix}.
    \end{equation}
    We will thus use $T^{k,\alpha}$ and $\mathcal{T}^{k,\alpha}$ interchangeably.
\end{proposition}
\begin{proof}
    Following the proof of \cite[Proposition 2.1]{feppon.cheng.ea2022}, it is easy to infer that 
    \[
        \begin{pmatrix}
            T^{k}[f]_{i}^{\iR}, \\
            T^{k}[f]_{i+1}^{\iL}
        \end{pmatrix}
            =
            \begin{pmatrix}
        \frac{\dd w_f^{\alpha}}{\dd x}(x_i^{\iR}) \\
        -\frac{\dd w_f^{\alpha}}{\dd x}(x_{i+1}^{\iL})
    \end{pmatrix}
    =A^{k}(s_{i}) \begin{pmatrix}
        f_{i}^{\iR} \\
        f_{i+1}^{\iL}
    \end{pmatrix} \text{ for all }i\in\Z,
    \] 
    where we extend $(f_i^{\iLR})_{1\leq i\leq N}$ by quasiperiodicity. The values of $(\mathcal{T}^{k,\alpha}[f]_i^{\iR},\mathcal{T}^{k,\alpha}[f]_{i+1}^{\iL})$ follow, with $1\leq i\leq N-1$. Then, we obtain the following values at $x_1^{\iL}$ and $x_{N}^{\iR}$ by using the quasiperiodicity:
    \[
   \begin{pmatrix}
        \frac{\dd w_f^{\alpha}}{\dd x}(x_N^{\iR}) \\
        -\frac{\dd w_f^{\alpha}}{\dd x}(x_{N+1}^{\iL})
   \end{pmatrix}
    =
   \begin{pmatrix}
        \frac{\dd w_f^{\alpha}}{\dd x}(x_N^{\iR}) \\
        -e^{\i \alpha L}\frac{\dd w_f^{\alpha}}{\dd x}(x_{1}^{\iL})
   \end{pmatrix}
    =A^{k}(s_{i }) \begin{pmatrix}
        f_{i}^{\iR} \\
        f_{i+1}^{\iL}
    \end{pmatrix}
    =A^{k}(s_{N})\begin{pmatrix}
        f_N^{\iR} \\
        f_1^{\iL}e^{\i \alpha L}
    \end{pmatrix},
    \] 
    which yields \eqref{eqn:DTNexplicit}.
\end{proof}
Remarkably, the $2N\times 2N$ matrix associated to $\mathcal{T}^{k,\alpha}$ is Hermitian.
It can be verified that the solution $w_f^{\alpha}$ to \eqref{eqn:defDTN} with $k\neq 0$ converges as $k\rightarrow 0$ to the solution to the same equation with $k=0$. As it can be expected from the matrix representation \eqref{eqn:DTNexplicit}, the operator $\mathcal{T}^{k,\alpha}$ is analytic in a neighbourhood of  $k=0$. In all what follows, we denote by $r$ the convergence radius
\[
    r:=\frac{\pi}{\max_{1\leq i\leq N}s_{i }}.
\] 
We identify $\mathcal{T}^{k,\alpha}$ with the matrix $T^{k,\alpha}$ of
\eqref{eq: Matrix form DTN}.
\begin{corollary}
    \label{prop:expansion DTN}
    The Dirichlet-to-Neumann map $\mathcal{T}^{k,\alpha}$ can be extended to 
    a holomorphic $2N\times 2N$ matrix with respect to the wave number $k \in \C$ on the disk $\vert k\vert<r$.
    Therefore, there exists a family of $2N\times 2N$ matrices $(\mathcal{T}^\alpha_{2n})_{n\in\N}$ such that $\mathcal{T}^{k,\alpha}$ admits the following  convergent series representation:
    \begin{equation}
       \label{eqn:DTN expansion}
        \mathcal{T}^{k,\alpha} =  
        \sum_{n=0}^{+ \infty} k^{2n} \mathcal{T}^{\alpha}_{2n},\qquad \forall k\in\C
        \text{ with }|k|<r.
    \end{equation}
The  matrices $\mathcal{T}^{\alpha}_0$ and $\mathcal{T}^{\alpha}_2$ of this series explicitly read
    \begin{equation}
    \label{eqn:3ttvg}
    \mathcal{T}_0^{\alpha}=\begin{pmatrix}
        -\frac{1}{s_{N }}         &&&&& \frac{1}{s_{N }}e^{-\i \alpha L}  \\
        & A_{0}(s_{1 }) & & & & \\
        & & A_{0}(s_{2 }) & & & \\
        &  & &\ddots & & \\
        & & & &  A_{0}(s_{N-1}) &\\
        \frac{1}{s_{N }}e^{\i \alpha L}  & & & & & -\frac{1}{s_{N}} \\
    \end{pmatrix},
    \end{equation}
    \begin{equation}
    \mathcal{T}_2^{\alpha}= \begin{pmatrix}
        \frac{1}{3}s_{N } & & & & \frac{1}{6}s_{N }e^{-\i \alpha L}\\
        & A_2(s_{1 }) &  & & \\
        & & \ddots & & \\
        & & & A_2(s_{ N-1 }) & \\
        \frac{1}{6}s_{N }e^{\i \alpha L}& &  & & \frac{1}{3}s_{N }
    \end{pmatrix},
    \end{equation}
    where for any $\ell \in\R$, $A_{0}(\ell)$ and $A_2(\ell)$ are the $2\times 2$ matrices given by
    \begin{equation}
    A_{0}(\ell):=\begin{pmatrix}
        -1/\ell & 1/\ell \\
        1/\ell & -1/\ell
    \end{pmatrix},\qquad
        A_2(\ell):=\begin{pmatrix}
            \frac{\ell}{3}& \frac{\ell}{6} \\
            \frac{\ell}{6} & \frac{\ell}{3}
        \end{pmatrix}.
    \end{equation}
\end{corollary}
\begin{proof}
    The result is immediate by noticing that for a given $\ell>0$, the matrix $A^{k}(\ell)$ of
    \eqref{eqn:1lzi8} is analytic with respect to the parameter $k$ on the disk $\vert k\vert \ell < \pi$, and its components are even functions of $k$. The expressions for $\mathcal{T}_0^{\alpha}$ and $\mathcal{T}_2^{\alpha}$ follow by computing the Taylor series of $A^{k}(\ell)$.
\end{proof}
\begin{remark}
    The expression \eqref{eqn:3ttvg} for $\mathcal{T}_0^{\alpha}$ can be more conveniently stated in terms of its action on a vector $f\equiv (f_i^{\iLR})_{1\leq i\leq N}\in\C^{2N}$ as
    \begin{equation}
    \label{eqn:4mxjs}
        \left\{\begin{aligned}
            \mathcal{T}_0^{\alpha}[f]_1^{\iL}& = -\frac{1}{s_{N }}(f_1^{\iL}-e^{-\i \alpha
            L}f_N^{\iR}), &\\
            \mathcal{T}_0^{\alpha}[f]_{i}^{\iL} &= -\frac{1}{s_{ i-1 }}(f_i^{\iL}-f_{i-1}^{\iR}), &
            2\leq i\leq N, \\
            \mathcal{T}_0^{\alpha}[f]_i^{\iR} &= \frac{1}{s_{i }}(f_{i+1}^{\iL}-f_i^{\iR}), &
            1\leq i\leq N-1, \\
            \mathcal{T}_0^{\alpha}[f]_N^{\iR} &= \frac{1}{s_{N }}(e^{\i\alpha
            L}f_1^{\iL}-f_N^{\iR}),& 
        \end{aligned}\right.
    \end{equation}
    or even more simply
    \[
        \left\{\begin{aligned}
            \mathcal{T}^{k,\alpha}[f]_{i}^{\iL}&=-\frac{1}{s_{ i-1 }}(f_i^{\iL}-f_{i-1}^{\iR})+\frac{k^{2}s_{i-1 }}{3}\left(f_i^{\iL}+\frac{1}{2}f_{i-1}^{\iR}\right)+O(k^{4}),\\
            \mathcal{T}^{k,\alpha}[f]_{i}^{\iR}&=\frac{1}{s_{i }}(f_{i+1}^{\iL}-f_i^{\iR})
            +\frac{k^{2}s_{i }}{3}\left(f_{i}^{\iR}+\frac{1}{2}f_{i+1}^{\iL}\right)
            +O(k^{4}),
        \end{aligned}\right. \qquad \forall i\in\Z,
    \]
    for any quasiperiodic sequence $(f_i^{\iLR})_{i\in\Z}\in\C^{2N,\alpha}$.
\end{remark}

\section{Subwavelength resonances}\label{sec: subwavelength resonances}
The one-dimensional problem \eqref{eq: system of coupled equations quasiperiodic} can be rewritten in terms of the Dirichlet-to-Neumann map as a set of coupled ordinary differential equations posed on the periodic segments $\crystal$:
\begin{equation}
    \label{eqn:DTN probelm}
    \left \{
    \begin{aligned}
        \left(\frac{\dd^2}{\dd x^2}+\frac{\omega^2}{v_i^2}\right) u(x,\alpha) &=0, & x\in D_i+L\Z,\\
        -\frac{\dd u}{\dd x}(x_i^{\iL}) &= \delta
        \mathcal{T}^{\frac{\omega}{v},\alpha}[u]_i^{\iL}  & \text{ for all } i\in\Z,\\
        \frac{\dd u}{\dd x}(x_i^{\iR}) &= \delta
        \mathcal{T}^{\frac{\omega}{v},\alpha}[u]_i^{\iR}  & \text{ for all } i\in\Z,\\
        u(x+L) &= u(x)e^{\i\alpha L}& \text{ for almost every }x\in\R,
    \end{aligned}
    \right.
\end{equation}
where for a function $u\in H^{1}_{\per,\alpha}(\R)$, we use the notation
$\mathcal{T}^{\frac{\omega}{v},\alpha}[u]\equiv
\mathcal{T}^{\frac{\omega}{v}}[(u(x_i^{\iLR}))_{i\in\Z}]$. 
\begin{definition}
    Any frequency $\omega^{\alpha}(\delta)$ such that \eqref{eqn:DTN probelm} admits a nontrivial solution $u$
is called a \emph{scattering resonance} \cite{zworski2017}. \emph{Subwavelength resonances} are those which in addition satisfy
\[
    \omega^{\alpha}(\delta)\rightarrow 0 \text{ as }\delta\rightarrow 0.
\] 
 The associated nontrivial solution $u^{\alpha}(\omega^{\alpha}(\delta),\delta)$ is called a \emph{subwavelength resonant mode}.
\end{definition}

\subsection{A first characterisation of subwavelength resonances based on an explicit representation of the solution}
\label{subsec:characterization}
Let us first state a characterisation of the subwavelength resonances which relies on a finite dimensional parametrisation of the solution $u$.
\begin{lemma}\label{lemma:A(ab)=0_means_scattering}
The subwavelength scattering resonances $\omega$ to
 the wave problem \eqref{eqn:DTN probelm} are the solution to the $2N\times 2N$ nonlinear eigenvalue problem 
    \begin{equation}
    \label{eqn:gzycv}
    \mathcal{A}^{\alpha}(\omega,\delta)\begin{pmatrix}
        a_i \\ b_i 
    \end{pmatrix}_{1\leq i\leq N}
    =0,
    \end{equation}
where $\mathcal{A}^{\alpha}(\omega,\delta)$ is the $2N\times 2N$ matrix given by
\begin{equation}
\label{eqn:ax3vi}
    \mathcal{A}(\omega,\delta):=   \i \diag\left(k_i\begin{pmatrix}
    -e^{\i k_i x_i^{\iL}} & e^{-\i k_i x_i^{\iL}} \\
    e^{\i k_i x_i^{\iR}} & -e^{-\i k_i x_i^{\iR}}
    \end{pmatrix}\right)_{1\leq i\leq N}
    -\delta \mathcal{T}^{\frac{\omega}{v},\alpha} \times \diag\left( \begin{pmatrix}
        e^{\i k_i x_i^{\iL}} & e^{-\i k_i x_i^{\iL}} \\
        e^{\i k_i x_i^{\iR}} & e^{-\i k_i x_i^{\iR}}
    \end{pmatrix} \right)_{1\leq i\leq N},
\end{equation}
and where $\mathcal{T}^{\frac{\omega}{v},\alpha}$ is the $2N\times 2N$ matrix defined by
\eqref{eqn:DTNexplicit}.
    Furthermore, resonant modes to \eqref{eqn:DTN probelm} correspond to 
    $(a_i\ b_i)^{T}_{1<i\leq N}$ by the formula
    \begin{equation}
    \label{eqn:wszun}
    u(x)=a_i e^{\i k_i x}+ b_i e^{-\i k_i x},\qquad \forall x\in (x_i^{\iL},x_i^{\iR}).
    \end{equation}
\end{lemma}
\begin{proof}
    Any solution $u$ to \eqref{eqn:DTN probelm} can be written as \eqref{eqn:wszun} and the boundary condition of \eqref{eqn:DTN probelm} reads 
\[
    \pm \i k_i ( a_i e^{\i k_i x_i^{\iLR}}-b_i e^{-\i k_i x_i^{\iLR}})-\delta
\mathcal{T}^{\frac{\omega}{v},\alpha}[u]^{\iLR}_i=0,
\] 
    which can be rewritten as \eqref{eqn:gzycv}.
\end{proof}
\begin{remark}
    With the characterisation of \Cref{lemma:A(ab)=0_means_scattering}, we reduce the spectral problem \eqref{eqn:DTN probelm} to a nonlinear finite-dimensional eigenvalue problem \eqref{eqn:gzycv}. We will exploit this property in the numerical computations. 
\end{remark}
\subsection{Characterisation of the subwavelength resonances based on the Dirichlet-to-Neumann map}
Multiplying by a test function $v\in H^{1}(D)$ and integrating on all the intervals
$(x_i^{\iL},x_i^{\iR})$, \eqref{eqn:DTN probelm} can be rewritten in the following weak form: find a nontrivial $u\in H^1(D)$ such that for any $v \in
H^1(D)$,
\begin{equation}
   \label{eqn:variational form}
    a^{\alpha}(u,v)=0,
\end{equation} 
where $a^{\alpha}$ is the bilinear form on $H^{1}(D)\times H^{1}(D)$ defined by
\[
    a^{\alpha}(u,v):=
\sum_{i=1}^N \int_{x_i^{\iL}}^{x_i^{\iR}} \left(\frac{\dd u}{\dd x} \frac{\dd \overline{v}}{\dd
       x} -\frac{\omega^2}{v_i^2} u \overline{v}\right) \dd x  - \sum_{i=1}^N \delta
       \left[\overline{v}(x_i^{\iR}) \mathcal{T}^{\frac{\omega}{v},\alpha}[u]_i^{\iR} +\overline{v}(x_i^{\iL})
       \mathcal{T}^{\frac{\omega}{v},\alpha}[u]_{i}^{\iLR}\right].
\] 
Following \cite{feppon.cheng.ea2022}, we introduce a new bilinear form $a^{\alpha}_{\omega,\delta}$ on $H^1(D) \times H^1(D)$:
\begin{multline}
    \label{eqn:def bilinear}
     a^{\alpha}_{\omega,\delta} (u,v) :=
     \sum_{i=1}^N \left[\int_{x_i^{\iL} }^{x_i^{\iR}} \frac{\dd u}{\dd x} \frac{\dd
     \overline{v}}{\dd x}\dd x
    + \int_{x_i^{\iL}}^{x_i^{\iR}} u \dd x \int_{x_i^{\iL}}^{x_i^{\iR}} \overline{v}\dd x\right] \\
    - \sum_{i=1}^{N}\left[
      \frac{\omega^{2}}{v_i^{2}} \int_{x_i^{\iL}}^{x_i^{\iR}} u \overline{v} \dd x  + \delta
     [\overline{v}(x_i^{\iR}) \mathcal{T}^{\frac{\omega}{v},\alpha}[u]_i^{\iR} + \overline{v}(x_i^{\iL})
     \mathcal{T}^{\frac{\omega}{v},\alpha}[u]_i^{\iL} ] \right].
\end{multline}
The bilinear form $a^{\alpha}_{\omega,\delta} (u,v)$ is obtained by adding the rank-one bilinear
forms $(u,v) \to \int_{x_i^{\iL}}^{x_i^{\iR}} u \dd x \int_{x_i^{\iL}}^{x_i^{\iR}}
\overline{v}\dd x$ to the bilinear form $a^{\alpha}$. Clearly, $a^{\alpha}_{\omega,\delta}$ is an
analytic perturbation in $\omega$ and $\delta$ of the bilinear form $a_{0,0}$ defined
by
\[ a_{0,0}(u,v) =  \sum_{i=1}^N \left[\int_{x_i^{\iL}}^{x_i^{\iR}} \frac{\dd u}{\dd x} \frac{\dd
\overline{v}}{\dd x}\dd x +  \int_{x_i^{\iL}}^{x_i^{\iR}} u \dd x \int_{x_i^{\iL}}^{x_i^{\iR}}
\overline{v}\dd x\right],\]
which is continuous coercive on $H^1(D)$. From standard perturbation theory, $a^{\alpha}_{\omega, \delta}$ remains coercive for sufficiently small complex values of $\omega$ and $\delta$. 

In order to characterise the subwavelength resonant modes, it is useful to introduce $h_j^{\alpha}(\omega,\delta)$ the solution to the variational problems
\begin{equation}
\label{eqn:9d5p8}
    a^{\alpha}_{\omega,\delta}(h_j^{\alpha}(\omega,\delta),v)=\int_{x_j^{\iL}}^{x_j^{\iR}}\bar v\dd x,\quad
    \forall v\in H^{1}(D), \qquad
   \forall 1\leq j\leq N.
\end{equation}
In the following lemma we show that the functions $h_j^{\alpha}(\omega,\delta)$ allow to reduce \eqref{eqn:variational form} to a finite dimensional $N\times N$ linear system by acting as basis functions.



\begin{lemma}\label{lemma: admits sol if I-Cx=0}
Let $\omega \in \C$ and $\delta \in \R$ belong to a neighbourhood of zero such that $a^{\alpha}_{\omega,\delta}$ is coercive. The variational problem \eqref{eqn:variational form} admits a nontrivial solution $u\equiv u(\omega,\delta)$ if and only if the $N\times N$ nonlinear eigenvalue problem
    \begin{equation}
    \label{eqn:lr8u3}
        (I-\exactC(\omega,\delta))\bm x = 0
    \end{equation}
    has a  solution $\omega$ and $\bm x:=(x_i(\omega,\delta))_{1\leq i\leq N}$, where $\exactC(\omega,\delta)$ is the matrix  given by
    \begin{equation}
    \label{eqn:hnkcq}
        \exactC(\omega,\delta)\equiv (\exactC(\omega,\delta)_{ij})_{1\leq i,j\leq N}:= 
        \left( \int_{x_i^{\iL}}^{x_i^{\iR}} h_j^{\alpha}(\omega,\delta)\dd x
        \right)_{1\leq i,j\leq N}.   \end{equation}
      When it is the case, $\omega$ is a subwavelength resonance and an associated resonant mode $u^{\alpha}(\omega,\delta)$ solution to \eqref{eqn:variational form} (equivalently, to \eqref{eq: system of coupled equations quasiperiodic} and \eqref{eqn:DTN probelm}) reads
 \begin{equation}
 \label{eqn:w0ion}
     u^{\alpha}(\omega,\delta) = \sum_{j=1}^{N} x_j(\omega,\delta)
     h_j^{\alpha}(\omega,\delta)
 \end{equation}
    with $h_j^{\alpha}(\omega,\delta)$ being defined by \eqref{eqn:9d5p8}.
\end{lemma}

\begin{proof}
    The variational problem \eqref{eqn:variational form} reads equivalently
    \begin{equation}
    \label{eqn:q48jq}
\begin{aligned}
       a^{\alpha}(u,v)=0        & \Leftrightarrow 
    a^{\alpha}_{\omega,\delta}(u,v)-\sum_{i=1}^{N} \left(\int_{x_i^{\iL}}^{x_i^{\iR}}u\dd x\right)
       a^{\alpha}_{\omega,\delta}(u_i^{\alpha},v)=0
        \\
    &        \Leftrightarrow 
        u-\sum_{i=1}^{N}\left(\int_{x_i^{\iL}}^{x_i^{\iR}}u\dd x
        \right)u_i^{\alpha}=0.
\end{aligned}
    \end{equation}
    By integrating both sides of \eqref{eqn:q48jq} on $(x_i^{\iL},x_i^{\iR})$, we find that the vector $\bm x:=\left(\int_{x_i^{\iL}}^{x_i^{\iR}} u(\omega,\delta)\dd x\right)_{1\leq i\leq N}$ solves the linear system
    \[ \int_{x_i^{\iL}}^{x_i^{\iR}} u(\omega,\delta) 
    \dd x -  \sum_{j=1}^N  \int_{x_i^{\iL}}^{x_i^{\iR}} h_j^{\alpha}(\omega,\delta) \dd x
    \int_{x_j^{\iL}}^{x_j^{\iR}} u(\omega,\delta) \dd x
    =0,\qquad 1\leq i\leq N,
    \]
    which is exactly \eqref{eqn:lr8u3}. Conversely, if \eqref{eqn:lr8u3} has a solution, then the second line of \eqref{eqn:q48jq} shows that the solution to \eqref{eqn:variational form} is given by \eqref{eqn:w0ion}.
\end{proof}
Subwavelength resonances are therefore the characteristic values $\omega\equiv\omega(\delta)$ for which $I-\exactC(\omega,\delta)$ is not invertible.  
\subsection{Asymptotic expansions of the subwavelength resonances}
\label{subsec:asympt}
We now show the existence of $N$ subwavelength resonances for any $\alpha\in Y^*$ and we compute their leading-order asymptotic expansions in terms of $\delta$.  We start by computing explicit asymptotic expansions of the functions $h_j^{\alpha}(\omega,\delta)$ solutions to \eqref{eqn:9d5p8}.
Here and hereafter, the characteristic function of a set $S$ is written as $\mathds{1}_{S}$.
\begin{proposition}\label{prop: asymptotic of uj}
   \label{prop: expansion of Id functions}
    Let $\omega \in \C$ and $\delta \in \R$ belong to a small enough neighbourhood of zero. The unique solution $h_j(\omega, \delta)$ with $1\leq j\leq N$ to the variational problem \eqref{eqn:9d5p8} has the following asymptotic behaviour as $\omega,\delta \to 0$:
   \begin{equation}
   \label{eqn:xxc8r}
    \begin{aligned}
        h_j(\omega,\delta)&= 
    \left(\frac{1}{\ell_j}+\frac{\omega^{2}}{v_j^{2}\ell_j^{2}}\right)\mathds{1}_{(x_j^{\iL},x_j^{\iR})}
        \\
        & 
    +
        \delta\left[\frac{\mathds{1}_{\{2,\dots, N\}}(j)}{\ell_{j-1}^{2}\ell_j}\frac{1}{s_{ j-1 }}\mathds{1}_{(x_{j-1}^{\iL},x_{j-1}^{\iR})}
        -\frac{1}{\ell_j^{3}}\left(\frac{1}{s_{ j-1 }}+\frac{1}{s_{j }}\right)\mathds{1}_{(x_j^{\iL},x_j^{\iR})}
        +\frac{\mathds{1}_{\{1,\dots,N-1\}}(j)}{\ell_j\ell_{j+1}^{2}s_{j }}\mathds{1}_{(x_{j+1}^{\iL},x_{j+1}^{\iR})}
        \right.\\
        &\left.+\vphantom{\frac{\mathds{1}_{\{2,\dots, N\}}(j)}{\ell_{j-1}^{2}\ell_j}}
        \frac{\delta_{1j}}{\ell_N^{2}\ell_1}\frac{e^{-\i \alpha L}}{s_{N}}\mathds{1}_{(x_N^{\iL},x_N^{\iR})}
        +\frac{\delta_{Nj}}{\ell_1^{2}\ell_N}\frac{e^{\i\alpha
        L}}{s_{N}}1_{(x_1^{\iL},x_1^{\iR})}+\widetilde{h}_{j,0,1}\right]
                +O((\omega^{2}+\delta)^{2}),
    \end{aligned}
   \end{equation}
    where $\widetilde{h}_{j,0,1}$ is some (quadratic) functions satisfying 
    \[
        \int_{x_i^{-}}^{x_i^{+}}\widetilde{h}_{j,0,1}\dd x = 0, \qquad \forall
        1\leq i\leq N.
    \] 
\end{proposition}

\begin{proof}
    From the definition of $a^{\alpha}_{\omega, \delta}$, the function $h_j^{\alpha}\equiv
    h_j^{\alpha}(\omega,\delta)$ satisfies the
    following differential equation written in strong form:
   \begin{equation}
   \label{eqn:ID functions}
    \left\{
    \begin{aligned}
     -\frac{\dd^2}{\dd x^2}h_j^{\alpha} - \frac{\omega^2}{v_b^2} h_j^{\alpha} + \sum_{i=1}^{N}
        \left(\int_{x_i^{\iL}}^{x_i^{\iR}} h_j^{\alpha} \dd x\right) \mathds{1}_{(x_i^{\iL},x_i^{\iR})} &= 
        \mathds{1}_{(x_j^{\iL},x_j^{\iR})} &  \text{ in }
        \bigsqcup_{i=1}^{N}(x_i^{\iL},x_i^{\iR}),\\
        -\frac{\dd h_j^{\alpha}}{\dd x}(x_i^{\iL}) &= \delta
        \mathcal{T}^{\frac{\omega}{v},\alpha}[h_j^{\alpha}]_i^{\iL} & \text{ for all } 1 \leq i \leq
        N,\\
        \\
        \frac{\dd h_j^{\alpha}}{\dd x}(x_i^{\iR}) &= \delta
        \mathcal{T}^{\frac{\omega}{v},\alpha}[h_j^{\alpha}]_i^{\iR} & \text{ for all } 1 \leq i \leq
        N.\\
    \end{aligned}
    \right.
    \end{equation}
    Since $\mathcal{T}^{\frac{\omega}{v},\alpha}$ is analytic in $\omega^{2}$, it
    follows that
    $h_j^{\alpha}(\omega,\delta)$
    is analytic in $\omega^{2}$ and $\delta$: there exist functions $(h_{j,2p,k})_{p{\geq}0,
    k{\geq}0}$ such that $h_j^{\alpha}(\omega,\delta)$ can be written as 
    the following convergent series in $H^1(D)$:
    \begin{equation}
        h_j^{\alpha}(\omega,\delta) = \sum_{p,k=0}^{+\infty} \omega^{2p}\delta^k
        h_{j,2p,k}.
    \end{equation}
    By using \Cref{prop:expansion DTN} and identifying powers of $\omega$ and
    $\delta$, we obtain the following equations characterizing the functions
    $(h_{j,2p,k})_{p{\geq}0, k{\geq}0}$:
   \begin{align}
   \begin{dcases}
       -\frac{\dd^2}{\dd x^2}h_{j,2p,k}  + \sum_{i=1}^{N}\left(\int_{x_i^{\iL}}^{x_i^{\iR}}
        h_{j,2p,k} \dd x\right)
    \mathds{1}_{(x_i^{\iL},x_i^{\iR})} =  \frac{1}{v_j^2} h_{j,2p-2,k} +
        \mathds{1}_{(x_j^{\iL},x_j^{\iR})}  \delta_{0p} \delta_{0k}  & \text{ in }
      D,
        \\ -\frac{\dd
        h_{j,2p,k}}{\dd x}(x_i^{\iL}) = \sum_{n=0}^p \frac{1}{v^{2n}}
        \mathcal{T}_{2n}^{\alpha}[h_{j,2p-2n,k-1}]_i^{\iL}, &1 \leq i \leq N,\\
        \frac{\dd
        h_{j,2p,k}}{\dd x}(x_i^{\iR}) = \sum_{n=0}^p \frac{1}{v^{2n}}
        \mathcal{T}_{2n}^{\alpha}[h_{j,2p-2n,k-1}]_i^{\iR}, &1 \leq i \leq N,\\
        \end{dcases}
   \end{align}    
   with the convention  that $h_{j,2p,k}=0$ for negative indices $p$ and $k$. 
    It can then be easily obtained by induction that 
    \[ h_{j,2p,0} = \frac{\mathds{1}_{(x_j^{\iL},x_j^{\iR})}}{v_j^{2p} \ell_j^{p+1}} \text{ for
    any } p \geq 0,\qquad 1\leq j\leq N. \]
   Then, for $p=0$ and $k=1$, we find that $h_{j,0,1}$ satisfies
   \begin{equation}
   \label{eqn:e5jdc}\left\{
   \begin{aligned}
     -\frac{\dd^2}{\dd x^2}h_{j,0,1}  +\sum_{i=1}^{N} \left(\int_{x_i^{\iL}}^{x_i^{\iR}}
       h_{j,0,1} \dd x\right) \mathds{1}_{(x_i^{\iL},x_i^{\iR})} &= 0    & \text{ in }
       D,\\
    -\frac{\dd h_{j,0,1}}{\dd x}(x_i^{\iL}) &= \mathcal{T}_0^{\alpha}[h_{j,0,0}]_i^{\iL} & \text{ for all } 1 \leq i \leq N, \\
    \frac{\dd h_{j,0,1}}{\dd x}(x_i^{\iR}) &= \mathcal{T}_0^{\alpha}[h_{j,0,0}]_i^{\iR} & \text{ for all } 1 \leq i \leq N. \\
    \end{aligned}
    \right.
   \end{equation}
    From \eqref{eqn:4mxjs} with
    $f_{i}^{\iLR}:=h_{j,0,0}(x_i^{\iLR})=\delta_{ij}/\ell_j$ for $1\leq i\leq N$, we obtain
    \[
        \left\{\begin{aligned}
            \mathcal{T}_0[h_{j,0,0}]_1^{\iL}
            &=-\frac{1}{\ell_j}\frac{1}{s_{N}}(\delta_{1j}-\delta_{Nj}e^{-\i
            \alpha L}), &&\\
            \mathcal{T}_0[h_{j,0,0}]_i^{\iL} &
            =-\frac{1}{\ell_j}\frac{1}{s_{ i-1 }}\left(\delta_{ij}-\delta_{(i-1)j}\right)
            &&\text{
        for }2\leq i\leq N,\\
    \mathcal{T}_0[h_{j,0,0}]_i^{\iR}& =\frac{1}{\ell_j}\frac{1}{s_{i}}\left(\delta_{(i+1)j}-\delta_{ij}\right)
            &&            \text{ for }1\leq i\leq N-1,\\
            \mathcal{T}_0[h_{j,0,0}]_N^{\iR}
            &=\frac{1}{\ell_j}\frac{1}{s_{N}}\left( e^{\i \alpha
            L}\delta_{1j}-\delta_{Nj} \right) .&& 
        \end{aligned}\right.
    \]
    Here and in the rest of the text $\delta_{ij}$ is Kronecker delta. Multiplying \eqref{eqn:e5jdc} by $\mathds{1}_{(x_i^{\iL},x_i^{\iR})}$ and integrating by parts,
    we find that
    \begin{align*}
        \int_{x_i^{\iL}}^{x_i^{\iR}}h_{j,0,1}\dd x  & =\frac{1}{\ell_i }\left[ 
        \mathcal{T}_0^{\alpha}[h_{j,0,0}]_i^{\iL}+\mathcal{T}_0^{\alpha}[h_{j,0,0}]_i^{\iR}\right]
        \\
     & =\frac{1}{\ell_i\ell_j}
            \frac{1}{s_{i-1}}(\delta_{(i-1)j}-\delta_{ij}) \mathds{1}_{\{2,\dots,N\}}(i)
        +\frac{1}{\ell_i\ell_j}
        \frac{1}{s_{i}}(\delta_{(i+1)j}-\delta_{ij})\mathds{1}_{\{1,\dots,N-1\}}(i) \\
            & \quad 
            +\frac{1}{\ell_i\ell_j}\frac{1}{s_{N}}\left(\delta_{Nj}e^{-\i\alpha
            L}-\delta_{1j}\right)\delta_{i1}
            +\frac{1}{\ell_i\ell_j}\frac{1}{s_{N}}\left( e^{\i\alpha
            L}\delta_{1j}-\delta_{Nj} \right) \delta_{iN}.
    \end{align*}
    Isolating the different cases yields
    \begin{align*}
        \begin{dcases}
            \frac{1}{\ell_1\ell_N}\frac{1}{s_{N}}e^{-\i\alpha L} & \text{
            if }i=1,\,j=N,\\
            \frac{1}{\ell_{j-1}\ell_j}\frac{1}{s_{j-1}} & \text{ if
            }i=j-1,\,2\leq j\leq N,\\
            -\frac{1}{\ell_j^{2}}\left(\frac{\mathds{1}_{\{2,\dots,N\}}(j)}{s_{j-1}}
            +\frac{\mathds{1}_{\{1,\dots,N-1\}}(j)}{s_{j}}+\frac{\delta_{1j}+\delta_{jN}}{s_{N}}\right) & \text{ if }i=j,\\
            \frac{1}{\ell_j\ell_{j+1}}\frac{1}{s_{j}} & \text{ if
            }i=j+1,\,1\leq j\leq N-1,\\
            \frac{1}{\ell_1\ell_N}\frac{1}{s_{N}}e^{\i\alpha L} & \text{
            if }i=N,\,j=1.
        \end{dcases}
    \end{align*}
    Using Fredholm's alternative, this allows to infer that $h_{j,0,1}$ can be
    written as
\begin{multline}
h_{j,0,1}=
    \frac{\mathds{1}_{\{2,\dots,N\}}(j)}{\ell_{j-1}^{2}\ell_j}\frac{1}{s_{j-1}}\mathds{1}_{(x_{j-1}^{\iL},x_{j-1}^{\iR})}
        -\frac{1}{\ell_j^{3}}\left(\frac{1}{s_{j-1}}+\frac{1}{s_{j}}\right)\mathds{1}_{(x_j^{\iL},x_j^{\iR})}
        +\frac{\mathds{1}_{\{1,\dots,N-1\}}(j)}{\ell_j\ell_{j+1}^{2}}\frac{1}{s_{j}}
        \mathds{1}_{(x_{j+1}^{\iL},x_{j+1}^{\iR})}  \\
        +\frac{\delta_{1j}}{\ell_N^{2}\ell_1}\frac{e^{-\i \alpha L}}{s_{N}}\mathds{1}_{(x_N^{\iL},x_N^{\iR})}
        +\frac{\delta_{Nj}}{\ell_1^{2}\ell_N}\frac{e^{\i\alpha
        L}}{s_{N}}\mathds{1}_{(x_1^{\iL},x_1^{\iR})}+\widetilde{h}_{j,0,1},
\end{multline}
    where $\widetilde{h}_{j,0,1}$ is a function (in fact, a second order polynomial) satisfying $\int_{x_i^{\iL}}^{x_i^{\iR}}\widetilde{h}_{j,0,1}\dd x=0$ for any $1\leq i\leq N$, with the convention $s_{0}=s_{N}$. Furthermore, $\widetilde{h}_{j,0,1}$ is identically zero on $(x_i^{\iL},x_i^{\iR})$, where $i\notin \{j-1,j,j+1\}$.
%
%
\end{proof}

Next, we define the (quasiperiodic) capacitance matrix similar to the three-dimensional case \cite{ammari.davies.ea2020, ammari.davies.ea2021, ammari.davies.ea2022Exceptional}.
\begin{definition}[Quasiperiodic capacitance matrix]\label{def: Capacitance matrix}
    Consider the solutions $V_i^\alpha : \R\to \R$ of the problem
    \begin{align}
        \begin{dcases}
            -\frac{\dd^2}{\dd x^2} V_i^\alpha =0 & \R \setminus \crystal, \\
            V_i^\alpha(x)=\delta_{ij} & x\in D_j ,\\
            V_i^\alpha(x+mL) = e^{\i\alpha mL}V_i^\alpha(x) & m\in\Z;
        \end{dcases}
        \label{eq: def V_i}
    \end{align}
    {for $1\leq i,j\leq N$}. Then the capacitance matrix is defined coefficient-wise by
    \begin{align*}
        \capmati{ij} = -\int_{\partial D_i}\frac{\partial V_j^\alpha}{\partial \nu}\dd \sigma, 
    \end{align*}
    where $\nu$ is the outward-pointing normal.
\end{definition}
\begin{lemma}
    The capacitance matrix is given by
    \begin{align*}
        \capmati{ij}\coloneqq
    -\frac{1}{s_{ j-1 }}\delta_{i(j-1)}+\left(
    \frac{1}{s_{ j-1 }}+\frac{1}{s_{j }} \right)\delta_{ij}
    -\frac{1}{s_{j }}\delta_{i(j+1)} -\delta_{1j}\delta_{iN}\frac{e^{-\i
    \alpha L}}{s_{N }}-\delta_{1i}\delta_{jN}\frac{e^{\i
    \alpha L}}{s_{N }},
    \end{align*}
    that is,
    \begin{align*}
    \capmat = \begin{pmatrix}
        \frac{1}{s_N} + \frac{1}{s_1} & -\frac{1}{s_1} & & &  -\frac{e^{-\i \alpha L}}{s_N}\\
        -\frac{1}{s_1} & \frac{1}{s_1} + \frac{1}{s_2} & -\frac{1}{s_2} &\\
        & \ddots& \ddots&\ddots & \\
        & & \ddots& \ddots&  - \frac{1}{s_{N-1}} \\
        -\frac{e^{\i \alpha L}}{s_N}& & & - \frac{1}{s_{N-1}}& \frac{1}{s_{N-1}} + \frac{1}{s_N}
    \end{pmatrix}.
\end{align*}
\end{lemma}
\begin{proof}
    One notices that the solutions $V_i^\alpha$ to \eqref{eq: def V_i} for $1\leq i\leq N-1$ are given by
    \begin{align*}
        V_i(x)=\begin{dcases}
            \frac{1}{s_{i-1}}(x-x_i^{\iL}), & x_{i-1}^{\iR}\leq x \leq x_i^{\iL},\\
            1, & x_i^{\iL}\leq x \leq x_i^{\iR},\\
            -\frac{1}{s_{i}}(x-x_{i+1}^{\iL}), & x_{i}^{\iR}\leq x \leq x_{i+1}^{\iL},\\
            0&\text{else},
        \end{dcases}
    \end{align*}
while for bigger and smaller $i$ we multiply by the corresponding $e^{\i\alpha mL}$ factor. Derivation with respect to the outward-pointing normal and integrating on the boundary just means 
\begin{align}\label{eq: compact formula def cap mat}
    \capmati{ij} = - \left(-\left.\frac{\dd V_i^\alpha}{\dd x}\right\vert_{\iL}(x_j^{\iL}) + \left.\frac{\dd V_i^\alpha}{\dd x}\right\vert_{\iR}(x_j^{\iR})\right).
\end{align}
Evaluating \eqref{eq: compact formula def cap mat} concludes the proof.
\end{proof}

\begin{corollary}\label{cor: asympotic expansion C}
    We have the following asymptotic expansion for the matrix $\exactC(\omega,\delta)$ defined in
     \eqref{eqn:hnkcq}:
    \begin{equation}
     \label{eqn:expansion of A}   
      \exactC(\omega, \delta) =I + \omega^2 V^{-2} L^{-1} - \delta  L^{-1}\capmat L^{-1} 
 + O((\omega^2 +\delta)^2),
 \end{equation}
    where $L$ is the length matrix $L\coloneqq\diag((\ell_i))$ and $V\coloneqq\diag((v_i))$ the material parameter matrix.
\end{corollary}

\begin{proof}
    Integrating the asymptotic expansion \eqref{eqn:xxc8r} of $h_j(\omega,\delta)$ on the interval $(x_i^{\iL},x_i^{\iR})$, we obtain
   \begin{equation}
    \begin{aligned}
        \capmati{ij}(\omega,\delta) =& 
    \left(1+\frac{\omega^{2}}{v_i^{2}\ell_i}\right)\delta_{ij}
        \\ +
        &\delta\left[\frac{\mathds{1}_{\{1,\dots,N-1\}}(i)}{\ell_{i}\ell_j}\frac{1}{s_{ j-1 }}\delta_{i(j-1)}
        -\frac{1}{\ell_i\ell_j}\left(\frac{1}{s_{ j-1 }}+\frac{1}{s_{j }}\right)\delta_{ij}
        +\frac{\mathds{1}_{\{2,\dots,N\}}(i)}{\ell_j\ell_{i}}\frac{1}{s_{j }}
        \delta_{i(j+1)} \right. \\
        &
        \left.+\frac{\delta_{1j}\delta_{iN}}{\ell_N\ell_1}\frac{e^{-\i \alpha L}}{s_{N}}
        +\frac{\delta_{Nj}\delta_{1N}}{\ell_1\ell_N}\frac{e^{\i\alpha
        L}}{s_{N }}\right]
                +\BO((\omega^{2}+\delta)^{2}).
    \end{aligned}
   \end{equation}
        This yields the result.
\end{proof}
It is thus useful to introduce the \emph{generalised capacitance matrix}
\begin{align}
    \label{eq: def generalised capacitance matrix}
    \gencapmat\coloneqq V^2L\inv\capmat.
\end{align}
\begin{proposition}
   \label{prop: asymptotic expansion of freqs via eigenvals}
    Assume that the eigenvalues of $\gencapmat$ are simple. Then the $N$ subwavelength band functions
    $(\alpha \mapsto \omega^{\alpha}_i)_{1\leq i\leq N}$ satisfy to the first order
    \begin{align*}
        \omega_i^{\alpha} =  \pm \sqrt{\delta\lambda_i} + \BO(\delta),
    \end{align*}
    where $(\lambda_i^{\alpha})_{1\leq i\leq N}$ are the eigenvalues of the eigenvalue problem
\begin{equation}
\label{eq:eigevalue problem capacitance matrix}
\gencapmat\bm a_i = \lambda_i^{\alpha}\bm a_i,\qquad 1\leq i\leq N.
\end{equation}
We select the $N$ values of $\pm\sqrt{\delta\lambda_i}$ having positive real parts.
 \end{proposition}
 Remark that in the Hermitian case it is possible to reformulate \eqref{eq:eigevalue problem capacitance matrix} into a symmetric eigenvalue problem so that the eigenvalues are real.
 \begin{proof}
 From \Cref{lemma: admits sol if I-Cx=0}, we know that \eqref{eqn:DTN probelm} has a solution if and only if 
 \begin{align*}
     (I-\exactC(\omega,\delta))\bm x = 0
 \end{align*}
 for some nonzero $\bm x$. Applying the asymptotic expansion from \Cref{cor: asympotic expansion C}, we obtain that the above equation is equivalent to
 \begin{align*}
     &0=\omega^2V^{-2}L\inv\bm x - \delta L\inv \capmat \underbrace{L\inv\bm x}_{\coloneqq \bm y} + \BO((\omega^{2}+\delta)^{2})\\
     &\Leftrightarrow V^2L\inv \capmat \bm y = \frac{\omega^2}{\delta}\bm y + \BO((\omega^{2}+\delta)^{2}),
 \end{align*}
     meaning that $\frac{\omega^2}{\delta}$ must be approximately an eigenvalue of $V^2L\inv \capmat=\gencapmat$.
 \end{proof}
 We refer to \cite[Proposition 3.7]{feppon.ammari2022} for a generalisation of \Cref{prop: asymptotic expansion of freqs via eigenvals}. The capacitance matrix provides also an approximation of the eigenmodes.
 \begin{lemma}\label{lemma: approx of eigemodes via eigenvectors}
     Let $u^\alpha$ be a subwavelength resonant eigenmode corresponding to $\omega^\alpha$ from \Cref{prop: asymptotic expansion of freqs via eigenvals}. Let $\bm a$ be the corresponding eigenvector of the generalised capacitance matrix. Then
     \begin{align*}
         u^\alpha(x) = \sum_j \bm a^{(j)}V_j^\alpha(x) + \BO(\delta),
     \end{align*}
     where $V_j^\alpha$ are the functions from \eqref{eq: def V_i} in \Cref{def: Capacitance matrix} and $\bm a^{(j)}$ denotes the $j$-th entry of the eigenvector.
 \end{lemma}
 \begin{proof}
    We sketch the proof, referring to \cite{feppon.ammari2022} for more details. We consider the case $N=2$ as $N>2$ is only notationally more difficult. Let $u^\alpha(x)$ be a resonant eigenmode. According to \Cref{lemma: admits sol if I-Cx=0}, we may represent the resonant mode (inside the resonators) as
    \begin{align*}
        u^\alpha(x) = \ell_1\bm a^{(1)} h_1^\alpha + \ell_2\bm a^{(2)} h_2^\alpha + \BO(\delta).
    \end{align*}
    Remark that we used the change of basis $L\inv$ in \Cref{prop: asymptotic expansion of freqs via eigenvals}, so the approximation of the $\bm x$ of \Cref{lemma: admits sol if I-Cx=0} is $L\inv \bm a$. The asymptotic expansion of \Cref{prop: asymptotic of uj} shows that $h_i^\alpha=\frac{1}{\ell_i}\mathds{1}_{D_i}+ \BO(\delta)$, so that we get the result inside the resonators.

    In order to obtain a solution outside, we may apply \Cref{def:DTN}. Expanding the result of \Cref{def:DTN} for small $\delta$, we obtain a linear interpolation between the boundary points, that is $V_j^\alpha$ outside of the resonators.
 \end{proof}
\section{Hermitian case}\label{sec: Hermitian}
In this section, we analyse in closer detail the Hermitian case. For simplicity, we consider the case when $v_i=v_b$ for all $i$ for some $v_b\in\R_{>0}$. One remarks that in this case the eigenvalue problem \eqref{eq:eigevalue problem capacitance matrix} may be simplified by finding eigenvalues of $L\inv \capmat$ and multiplying the eigenvalues by $v_b^2$.

In the general case, we have to solve the generalised eigenvalue problem 
\begin{align}\label{eq: generalised eigenvalue problem}
    \capmat \bm a_i = v_b^{-2}\lambda_i L \bm a_i.
\end{align}
After a change of basis, we recover a symmetric eigenvalue problem having the same eigenvalues as \eqref{eq: generalised eigenvalue problem}
\begin{align}\label{eq: generalised eigenvalue problem symmetric}
    L^{-\frac{1}{2}} \capmat L^{-\frac{1}{2}} \bm b_i = v_b^{-2}\lambda_i \bm b_i.
\end{align}
From \eqref{eq: generalised eigenvalue problem symmetric}, we see that in the Hermitian case the subwavelength resonances are real.

\subsection{Dirac degeneracy and Zak phase}
We first prove the following result. 
 \begin{lemma}
     The eigenspace associated to $C^{\alpha}$ has dimension at most two.
 \end{lemma}
 \begin{proof}
     This is a consequence of the tridiagonal structure of $\capmat$: one can
     extract from $\capmat-\lambda_p^{\alpha} L$ a full rank minor of dimension $(N-2)\times
     (N-2)$ which is an upper triangular matrix with 
     diagonal $-\frac{1}{s_{1}}\dots -\frac{1}{s_{ N-2 }}$.
 \end{proof}
 The following lemma concerns degeneracies of the capacitance matrix.
 \begin{lemma}\label{lemma: double eigenvalue at pi/L}
     Assume that $N=2$. The only configuration such that \eqref{eq: generalised eigenvalue problem symmetric} admits a
     double eigenvalue is the one with $\ell_1=\ell_2$ and  $s_{1}=s_{2}$. Moreover, this double eigenvalue occurs at $\alpha=\pm
     \frac{\pi}{L}$, and $\mathcal{C}^{\pm \frac{\pi}{L}}=2s_{1}I$, where $I$ is the identity matrix.
 \end{lemma}
 \begin{proof}
     Problem \eqref{eq: generalised eigenvalue problem symmetric} reduces to find the eigenvalue of 
     \begin{equation}
     \label{eqn:7po4w}
     L^{-\frac{1}{2}}\capmat L^{-\frac{1}{2}}
     =\begin{pmatrix}
         \left(\frac{1}{s_{1}}+\frac{1}{s_{2}}\right)\ell_1^{-1}
         & 
         \left( -\frac{1}{s_{1}}-\frac{1}{s_{2}}e^{\i\alpha L}
         \right)\ell_1^{-\frac{1}{2}}\ell_2^{-\frac{1}{2}} \\
         \left( -\frac{1}{s_{1}}-\frac{1}{s_{2}}e^{-\i\alpha L}
         \right)\ell_1^{-\frac{1}{2}}\ell_2^{-\frac{1}{2}} & 
         \left(\frac{1}{s_{1}}+\frac{1}{s_{2}}\right)\ell_2^{-1}
     \end{pmatrix}.
     \end{equation}
     The characteristic polynomial of this matrix is
     \begin{align*}
         P(\lambda)&=\det\left( L^{-\frac{1}{2}}\capmat L^{-\frac{1}{2}}-\lambda I \right)\\
         &=\left( 
         \left(\frac{1}{s_{1}}+\frac{1}{s_{2}}\right)\ell_1^{-1}-\lambda \right)
         \left(
         \left(\frac{1}{s_{1}}+\frac{1}{s_{2}}\right)\ell_2^{-1}-\lambda
         \right)-
         \ell_1^{-1}\ell_2^{-1}\left| \frac{1}{s_{1}}+\frac{1}{s_{2}}e^{\i
         \alpha L} \right|^{2}\\
         &=\lambda^{2}-\left( \frac{1}{s_{1}}+\frac{1}{s_{2}} \right)\left(
         \ell_1^{-1}+\ell_2^{-1} \right)\lambda 
         +\ell_1^{-1}\ell_2^{-1}\left[ \left( \frac{1}{s_{1}}+\frac{1}{s_{2}}
         \right)^{2}-\left|  \frac{1}{s_{1}}+\frac{1}{s_{2}}e^{\i
         \alpha L} \right| \right].
     \end{align*} 
     Therefore, a multiple eigenvalue occurs when the discriminant of this second
     order polynomial vanishes, which is the case when
     \begin{align*}
         0 & =\left( \frac{1}{s_{1}}+\frac{1}{s_{2}} \right)^{2}\left(
         \ell_1^{-1}+\ell_2^{-1} \right)^{2}-4\ell_1^{-1}\ell_2^{-1}\left[ \left( \frac{1}{s_{1}}+\frac{1}{s_{2}}
         \right)^{2}-\left|  \frac{1}{s_{1}}+\frac{1}{s_{2}}e^{\i
         \alpha L} \right| \right] \\
             & = \left( \frac{1}{s_{1}}+\frac{1}{s_{2}} \right)^{2}\left(
         \ell_1^{-1}-\ell_2^{-1} \right)^{2}
             +4\ell_1^{-1}\ell_2^{-1}\left|  \frac{1}{s_{1}}+\frac{1}{s_{2}}e^{\i
         \alpha L} \right|^{2}.
    \end{align*}
     This readily implies $\ell_1=\ell_2$, and then
     \[
         \left\{\begin{aligned}
             \frac{1}{s_{1}}+\frac{1}{s_{2}}\cos(\alpha L) &= 0,\\
             \frac{1}{s_{2}}\sin(\alpha L) &= 0.
         \end{aligned}
         \right.
     \] 
     For this system to admit a solution with $0<s_{1},s_{2}$, it is
     necessary that $\alpha=\pm \frac{\pi}{L}$, and then we must have
     $s_{1}=s_{2}$.
 \end{proof}

Therefore, we study  the eigenvalues of $\capmat$ for regularly spaced dimers of resonators (i.e., $N=2$, $\ell_1=\ell_2$ and $s_1=s_2$). Let us rewrite \eqref{eqn:7po4w} only in terms of $s_{1}$ and $\ell_1$:
 \[
 L^{-\frac{1}{2}}\capmat L^{-\frac{1}{2}}=
\frac{2}{\ell_{1}s_{1}} \begin{pmatrix}
    1 & -e^{\i \alpha L/2}\cos(\alpha L /2) \\
     -e^{-\i \alpha L/ 2}\cos(\alpha L/ 2) & 
     1
 \end{pmatrix}.
 \] 
 The eigenvalues of this matrix are 
 \begin{align*}
     \lambda_0(\alpha)&=\frac{2}{\ell_1 s_{1}}\left(1-\cos\left( \frac{\alpha
     L}{2} \right)\right)=\frac{4}{\ell_{1}s_{1}}\sin^{2}\left(\frac{\alpha
     L}{4}\right),\\ 
     \lambda_1(\alpha)&=\frac{2}{\ell_{1}s_{1}}\left(1+\cos\left( \frac{\alpha
     L}{2} \right)\right)=\frac{4}{\ell_1s_{1}}\cos^{2}\left(\frac{\alpha
     L}{4}\right).
 \end{align*}
     
An associated family of eigenvectors read
 \[
 \begin{pmatrix}
     1 \\
     e^{-\i \frac{\alpha L}{2}}
 \end{pmatrix},
\\
 \begin{pmatrix}
     1 \\
    - e^{-\i \frac{\alpha L}{2}}
 \end{pmatrix}.
 \] 
Subwavelength resonances then read
\begin{align}
    \label{eq: asymptotic bands regularly spaced dimers}
    \omega_0^{\alpha}= \frac{2}{\sqrt{\ell_1s_{1}} } v_b \delta^{\frac{1}{2}} \left| \sin\left( \frac{\alpha
    L}{4} \right) \right| + \BO(\delta), \qquad
    \omega_1^{\alpha}= \frac{2}{\sqrt{\ell_1s_{1}} }v_b \delta^{\frac{1}{2}}\left| \cos\left( \frac{\alpha
    L}{4} \right) \right|+ \BO(\delta).
\end{align}
Hence, at leading order in $\delta$, a  band inversion occurs at $\alpha=\pm\frac{\pi}{L}$. Furthermore, \eqref{eq: asymptotic bands regularly spaced dimers} shows that at $\alpha=\pm\frac{\pi}{L}$ the bands form a \emph{Dirac degeneracy}. {The slopes of the two bands do not vanish at the Dirac degeneracy and satisfy $\frac{\dd}{\dd \alpha}\vert_{\pm\frac{\pi}{L}}\omega_0^\alpha=-\frac{\dd}{\dd \alpha}\vert_{\pm\frac{\pi}{L}}\omega_1^\alpha$}. Typically, breaking the symmetry of the structure results in the Dirac cone to open into a band gap \cite[Section 4]{ammari.fitzpatrick.ea2020}. We show this in \Cref{sec: Hermitian edge modes}.

The following lemma gives explicit formulas for the eigenvectors of the capacitance matrix.
\begin{lemma}\label{lemma: eigenpairs cap mat Hermitian}
    Assume that $\ell_1=\ell_2$. Then, the eigenvalues of the capacitance matrix are given by 
    \[
        \lambda_1^{\alpha}=\ell_{1}^{-1}\left[\left( \frac{1}{s_{1}}+\frac{1}{s_{2}}
        \right)-\left| \frac{1}{s_{1}}+\frac{1}{s_{2}}e^{\i
        \alpha L} \right|\right], \qquad
        \lambda_2^{\alpha}=\ell_{1}^{-1}\left[\left( \frac{1}{s_{1}}+\frac{1}{s_{2}}
        \right)+\left| \frac{1}{s_{1}}+\frac{1}{s_{2}}e^{\i
        \alpha L} \right|\right].
    \] 
    An the associated pair of eigenvectors is given by 
    \[
        a_1^{\alpha}= \frac{1}{\sqrt{2} }\begin{pmatrix}
            1 \\ e^{-\i \theta_{\alpha}} 
        \end{pmatrix},\qquad
        a_2^{\alpha}= \frac{1}{\sqrt{2} }\begin{pmatrix}
            1\\ -e^{-\i \theta_{\alpha}}
        \end{pmatrix},
    \] 
    where $\theta_{\alpha}$ is the argument such that
\begin{equation}
\label{eqn:ftrb8}
        -\left(\frac{1}{s_{1}}+\frac{1}{s_{2}}e^{\i \alpha L}\right)=\rho
        e^{\i \theta_{\alpha}}.
\end{equation}
\end{lemma}
\begin{definition}[Zak phase]\label{def: Zak Hermitian}
    For a non-degenerate band $\omega_j^\alpha$, we let $u_j^\alpha$ be a family of normalised eigenmodes which depend continuously on $\alpha$. Then we define the (Hermitian) Zak phase as
    \begin{align}
        \zak{j}\coloneqq \i\int_{Y^*}\left\langle u_j^\alpha,\frac{\partial}{\partial \alpha} u_j^\alpha \right\rangle \dd \alpha,
    \end{align}
    where $\langle\cdot,\cdot\rangle$ denotes the usual $L^2$ inner product.
\end{definition}
Using \Cref{lemma: eigenpairs cap mat Hermitian}, we obtain the Zak phase of the structure.
\begin{proposition}\label{prop: Zak Hermitian}
   Let $N=2$ and $\ell_1=\ell_2$. Then, we have 
    \[
        \zak{j}=\left\{\begin{aligned}
  \pi & \text{ if }s_{1}{\geq}s_{2}, \\
            0 & \text{ if }s_{1}< s_{2}.            
        \end{aligned}\right.
    \] 
\end{proposition}
One can prove \Cref{prop: Zak Hermitian} using a similar approach to the one in \cite{ammari.davies.ea2020}. We suggest a different proof whose presentation is postponed to \Cref{sec: non-Hermitian}, where it will result as a special case of the more general \Cref{thm: Zak non Hermitian}.
\subsection{Localised edge modes generated by geometrical defects}\label{sec: Hermitian edge modes}
In this subsection, we study an infinite structure composed by two periodic parts. We consider this structures as having a geometrical defect in the periodicity, see \Cref{fig: geometrical defect}.

Such structures have been studied in the case of tight-binding Hamiltonian systems \cite{drouot.fefferman.ea2020,fefferman.lee-thorp.ea2017,fefferman.lee-thorp.ea2014} and for an SSH chain of resonators in $\R^3$ \cite{ammari.davies.ea2020}.

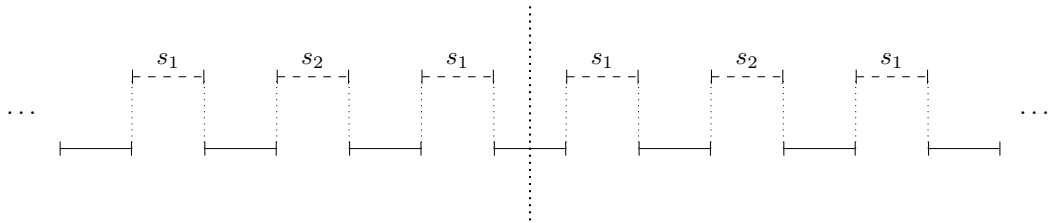
\begin{figure}[h]
    \centering
    \begin{adjustbox}{width=\textwidth}
        \begin{tikzpicture}
        \draw[-,thick,dotted] (-.5,-1) -- (-.5,2);
        \draw[|-|,dashed] (0,1) -- (1,1);
        \node[above] at (0.5,1) {$s_1$};
        \draw[|-|] (1,0) -- (2,0);
        \node[below] at (1.5,0) {};
        \draw[-,dotted] (1,0) -- (1,1);
        
        \draw[|-|,dashed] (2,1) -- (3,1);
        \node[above] at (2.5,1) {$s_2$};
        \draw[|-|] (3,0) -- (4,0);
        \node[below] at (3.5,0) {};
        \draw[-,dotted] (2,0) -- (2,1);
        \draw[-,dotted] (3,0) -- (3,1);
        \draw[-,dotted] (4,0) -- (4,1);
        
        \begin{scope}[shift={(+4,0)}]
        \draw[|-|,dashed] (0,1) -- (1,1);
        \node[above] at (0.5,1) {$s_1$};
        \draw[|-|] (1,0) -- (2,0);
        \node[below] at (1.5,0) {};
        \draw[-,dotted] (1,0) -- (1,1);
        \node at (2.5,.5) {\dots};
        \end{scope}
        
\begin{scope}[shift={(-4,0)}]

        \draw[|-|,dashed] (0,1) -- (1,1);
        \node[above] at (0.5,1) {$s_2$};
        \draw[|-|] (1,0) -- (2,0);
        \node[below] at (1.5,0) {};
        \draw[-,dotted] (1,0) -- (1,1);
        
        \draw[|-|,dashed] (2,1) -- (3,1);
        \node[above] at (2.5,1) {$s_1$};
        \draw[|-|] (3,0) -- (4,0);
        \node[below] at (3.5,0) {};
        \draw[-,dotted] (2,0) -- (2,1);
        \draw[-,dotted] (3,0) -- (3,1);
        \draw[-,dotted] (4,0) -- (4,1);
        \end{scope}
        
        \begin{scope}[shift={(-8,0)}]
        \node at (.5,.5) {\dots};
        \draw[|-|] (1,0) -- (2,0);
        \node[below] at (1.5,0) {};
        
        \draw[|-|,dashed] (2,1) -- (3,1);
        \node[above] at (2.5,1) {$s_1$};
        \draw[|-|] (3,0) -- (4,0);
        \node[below] at (3.5,0) {};
        \draw[-,dotted] (2,0) -- (2,1);
        \draw[-,dotted] (3,0) -- (3,1);
        \draw[-,dotted] (4,0) -- (4,1);
        \end{scope}
        
        \end{tikzpicture}
    \end{adjustbox}
    \caption{Infinite structure with a geometrical defect.}
    \label{fig: geometrical defect}
\end{figure}

The peculiarity of such defect structures is the support for edge modes. These modes have frequencies that lay in the band gap and thus are particularly robust with respect to perturbations. Furthermore, they are spatially localised near the defect.

To show the existence of an edge mode, we compute the subwavelength resonances of a finite but large array having the same geometrical defect. In the three-dimensional case, it has been shown that this is indeed an accurate approximation \cite{ammari.davies.ea2023Spectral}.
\Cref{fig: hermtian edgemode plot} shows the existence of edge modes. \Cref{fig: spectrum edgemode setup} illustrates that the frequencies of the edge modes are well-separated from the bulk and lay inside the band gap. This figure is of particular interest as it suggests that the spectrum of the finite approximation that does not lay in the band gap converges to the continuous spectrum of the periodic structure. In \Cref{fig: convergence relative error hermtian} we consider the frequencies in the band gap and compute a convergence scaling roughly as $\BO(n^{-10})$, where $n$ is the number of resonators in the structure. This exponential convergence is due to the fact that the dimension of the lattice is equal to that of the physical space \cite{ammari.davies.ea2023Spectral, lin.santosa2013Resonances, lin2016perturbation, lu.marzuola.ea2022Defect}. 

As mentioned before, edge frequencies laying in the band gap are typically robust to perturbations. In \Cref{fig: stability hermitian edge mode}, we show that these frequencies are only minimally influenced by slightly perturbing the distances between the resonators via
\begin{align*}
    \tilde{s_i} = s_i + \epsilon_i,\quad \epsilon_i\sim \mathcal{N}(0,\sigma^2)
\end{align*}
with $\mathcal{N}(0,\sigma^2)$ being a uniform distribution 
with standard deviation $\sigma$ and mean-value zero. 
In particular, they remain in the band gap. We thus call these edge modes \emph{topologically protected}.

\begin{figure}[h]
    \centering
    \begin{subfigure}[t]{0.48\textwidth}
    \centering
    \includegraphics[width=\textwidth]{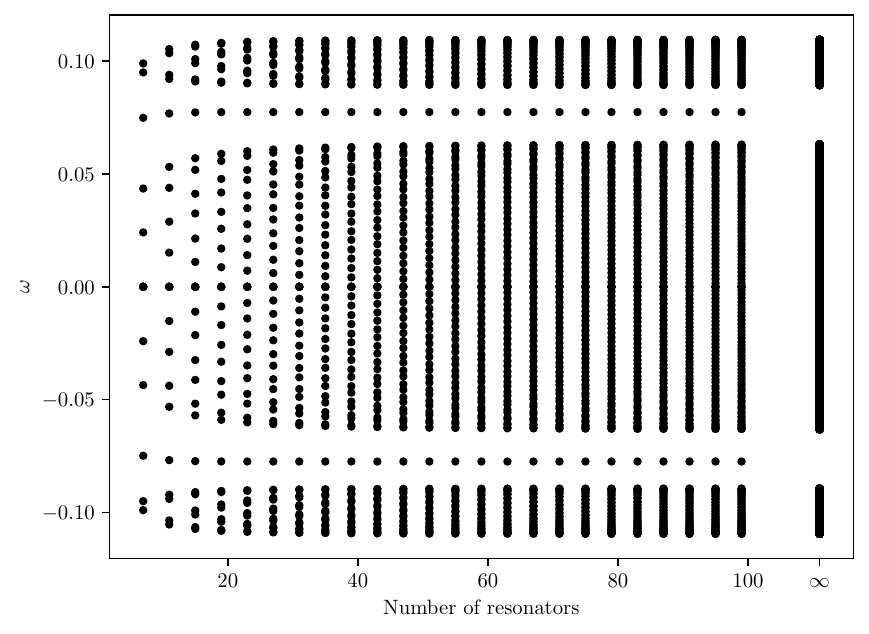}
    \caption{{Convergence of the subwavelength resonances of a finite structure with a geometrical defect for an increasing number of resonators. The solid line shows the spectrum of a periodic structure.}}
    \label{fig: spectrum edgemode setup}
    \end{subfigure}
    \hfill
    \begin{subfigure}[t]{0.48\textwidth}
        \centering
        \includegraphics[width=\textwidth]{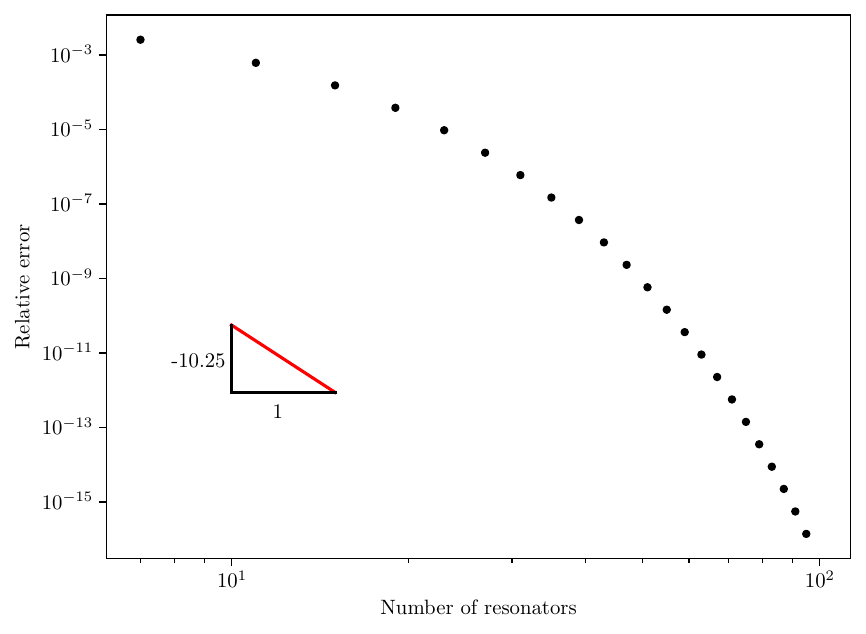}
        \caption{Convergence of the relative error $\abs{\omega_n-\omega_{100}}$ where $\omega_i$ is the edge subwavelength resonance in a finite structure with $i$ resonators. The plot is in log-log scale.}
        \label{fig: convergence relative error hermtian}
        \end{subfigure}
    \\[5mm]
\begin{subfigure}[t]{0.48\textwidth}
    \centering
    \includegraphics[width=\textwidth]{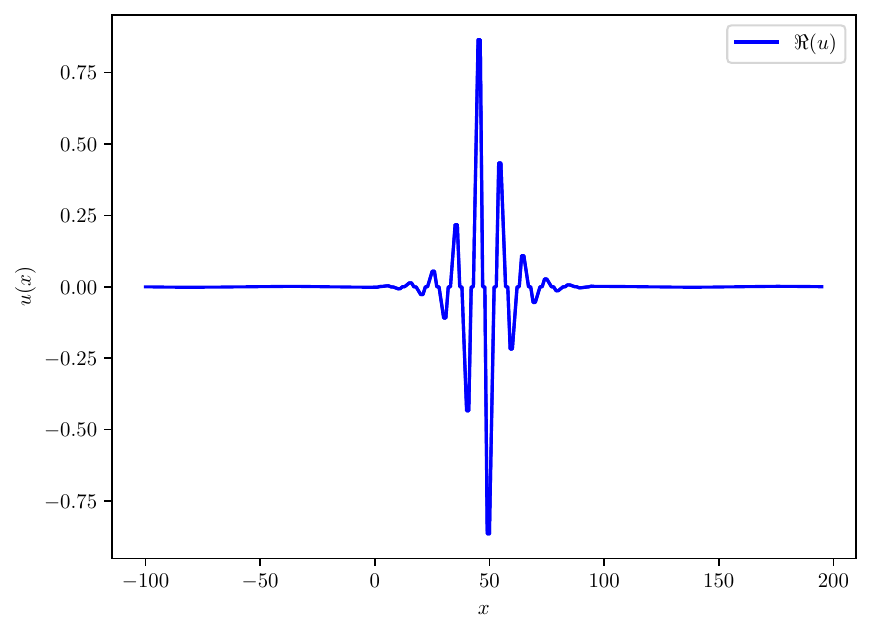}
    \caption{Localised edge mode for a finite but large array of $N=39$ resonators having a geometrical defect.}
    \label{fig: hermtian edgemode plot}

\end{subfigure}
\hfill
\begin{subfigure}[t]{0.48\textwidth}
    \centering
    \includegraphics[width=\textwidth]{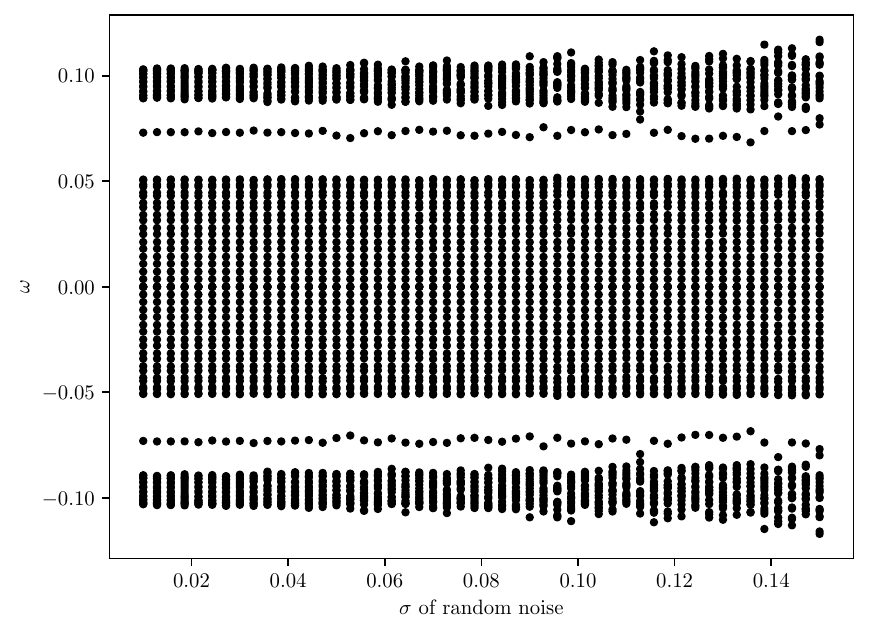}
    \caption{Stability of  subwavelength edge resonances with respect to perturbations in the geometry.}
    \label{fig: stability hermitian edge mode}
\end{subfigure}
\caption{Edge modes generated by geometrical defects. For the infinite structure, we use $N=2$, $\ell_i=1$, $s_1=2$, $s_2=1$.}
\label{fig:Edgemodes generated by geometrical defects}
\end{figure}

\section{Non-Hermitian case}\label{sec: non-Hermitian}
In the non-Hermitian case, the material parameters $\kappa_i$ are complex with non vanishing imaginary parts. As we want to analyse the influence of the complex material parameters, we assume for the rest of this section that the size of the resonators is constant, i.e.,  $\ell_i=\ell_1$ for all $1\leq i\leq N$.

A particular case of this non-Hermitian setup are systems with PT-symmetry. Originating from quantum mechanics, this terms defines a system where gains and losses are balanced, that is, $v_1 = \bar{v_2}$ in the case of a dimer of resonators.
\subsection{Non-Hermitian Zak phase}
\begin{definition}[Non-Hermitian Zak phase]\label{def: Zak phase Non Hermitian}
    The non-Hermitian Zak phase $\zak{j}$, for $1\leq j\leq N$, is defined by
    \begin{align*}
        \zak{j} \coloneqq \frac{\i}{2}\int_{Y^*} \left(\left\langle v_j^\alpha, \frac{\partial u_j^\alpha}{\partial \alpha} \right\rangle + \left\langle u_j^\alpha, \frac{\partial v_j^\alpha}{\partial \alpha} \right\rangle \right) \dd \alpha, 
    \end{align*}
    where $u_j^\alpha$ and $v_j^\alpha$ are respectively the left and right eigenmodes.
\end{definition}
We remark immediately that \Cref{def: Zak phase Non Hermitian} is a generalisation of \Cref{def: Zak Hermitian} as left and right eigenmodes are equal in the Hermitian case.

The following lemma is \cite[Lemma 3.5]{ammari.hiltunen2020}.
\begin{lemma}\label{lemma: approx Zak formula}
    Let $\bm u_j$ and $\bm v_j$ be a bi-orthogonal system (i.e., $\langle \bm v_i,\bm u_j\rangle=\delta_{ij}$) of eigenvectors of the generalised capacitance matrix defined by \eqref{eq: def generalised capacitance matrix}, so that \Cref{lemma: approx of eigemodes via eigenvectors} holds. Then, the Zak phase can be written as
    \begin{align}
        \zak{j}=-\Im\left(\int_{Y^*}\left\langle {\bm v}_j,\frac{\partial {\bm u}_j}{\partial \alpha} \right\rangle \dd\alpha\right)+\BO(\delta).
    \end{align}
\end{lemma}

We will now derive an explicit formula for the non-Hermitian Zak phase. This, as the non-Hermitian version is a generalisation of the Hermitian one, will allow us to prove \Cref{prop: Zak Hermitian}.
\begin{remark} \label{rmk: LA biorthogonal}
Consider an eigendecomposition
\begin{align*}
    M = U D U\inv
\end{align*}
of a matrix $M$, where $U$ is an invertible matrix with columns given by (right) eigenvectors and $D$ a diagonal matrix. Then, a basis of left eigenvectors is given by the columns of the matrix $V\coloneqq (U\inv)^*$. Furthermore, the two matrices are bi-orthogonal, meaning that $V^* U=I$ so that the left and right eigenvectors satisfy $\left\langle v_i,u_j\right\rangle = \delta_{ij}$.
\end{remark}
Let $U(\alpha)$ be an eigenbasis of the generalised quasiperiodic capacitance matrix and $V(\alpha)=(U(\alpha))\inv$ be the corresponding bi-orthogonal basis according to \Cref{rmk: LA biorthogonal}. Then, defining
\begin{align*}
    V^*(\alpha)\frac{\partial}{\partial \alpha} U(\alpha) = U\inv(\alpha)\frac{\partial}{\partial \alpha} U(\alpha)\eqqcolon
     J(\alpha),
\end{align*}
the Zak phase take the following form according to \Cref{lemma: approx Zak formula}:
\begin{align}\label{eq: zak as integral of product of matrices}
    \zak{j} = -\Im\left(\int_{Y^*} J_{j,j}(\alpha) \dd \alpha\right)+\BO(\delta).
\end{align}
Let $a=\frac{1}{s_{1}}+\frac{1}{s_{2}}$ and $b(\alpha)=-\frac{1}{s_{1}}-\frac{e^{-\i L \alpha}}{s_{2}}$, so that the generalised capacitance matrix is given by
\begin{align*}
    \gencapmat\coloneqq \begin{pmatrix}
        v_1^2a & v_1^2b(\alpha)\\
        v_2^2\bar{b(\alpha)} & v_2^2a
    \end{pmatrix}
\end{align*}
with eigenbasis given by the columns of
\begin{align}\label{eq: U(alpha) gen cap mat}
    U(\alpha)&\coloneqq \begin{pmatrix}
    -a(v_2^2-v_1^2)-\sqrt{a^2(v_1^2-v_2^2)^2+4v_1^2v_2^2\abs{b(\alpha)}^2} & -a(v_2^2-v_1^2)+\sqrt{a^2(v_1^2-v_2^2)^2+4v_1^2v_2^2\abs{b(\alpha)}^2} \\
    2 v_2^2\bar{b(\alpha)} & 2 v_2^2\bar{b(\alpha)}
    \end{pmatrix}\\
    &\eqqcolon\begin{pmatrix}
    -a(v_2^2-v_1^2)-\sqrt{f(b(\alpha))} & -a(v_2^2-v_1^2)+\sqrt{f(b(\alpha))} \nonumber\\
    2 v_2^2\bar{b(\alpha)} & 2 v_2^2\bar{b(\alpha)}
    \end{pmatrix}.
\end{align}
Actually, if $b(\alpha)=0$, then this formula for $U(\alpha)$ does not work. However, as we will be later interested in integrating this quantity and the set $\{\alpha:b(\alpha)=0\}$ has zero measure, we can just work with the formula above.

In particular, for a non-degenerate $\gencapmat$, we have
\begin{align*}
    U(\alpha)\inv = \frac{1}{4v_2^2\bar{b(\alpha)}f(b(\alpha))} \begin{pmatrix}
    2 v_2^2\bar{b(\alpha)} & +a(v_2^2-v_1^2)-\sqrt{f(b(\alpha))} \\
    -2 v_2^2\bar{b(\alpha)} & -a(v_2^2-v_1^2)-\sqrt{f(b(\alpha))}
    \end{pmatrix},
\end{align*}
so that
\begin{align*}
    J_{1,1} = \frac{2 v_2^2\bar{b(\alpha)}}{4v_2^2\bar{b(\alpha)}\sqrt{f(b(\alpha))}}\frac{\partial}{\partial \alpha}(-\sqrt{f(b(\alpha))}) + \frac{a(v_2^2-v_1^2)-\sqrt{f(b(\alpha))}}{4v_2^2\bar{b(\alpha)}\sqrt{f(b(\alpha))}}\frac{\partial}{\partial \alpha}2 v_2^2\bar{b(\alpha)}.
    %
\end{align*}
By periodicity, we know that $b(\alpha)$ draws a closed path in $\C$. Remark that $f(b(\alpha))$ is a closed curved tracing a line (or two segments), so that integrating over it always results in zero. Reformulating the above in terms of path integral we get
\begin{align*}
    \int_{Y^*} J_{1,1}(\alpha)\dd \alpha = &-\frac{1 }{2\sqrt{f(b(\alpha))}}\frac{\partial}{\partial \alpha}(\sqrt{f(b(\alpha))})\dd \alpha + \frac{a(v_2^2-v_1^2)}{2}\int_{Y^*}\frac{1}{\bar{b(\alpha)}\sqrt{f(b(\alpha))}}\frac{\partial}{\partial \alpha}\bar{b(\alpha)}\dd \alpha \\&- \frac{1}{2}\int_{Y^*}\frac{1}{\bar{b(\alpha)}}\frac{\partial}{\partial \alpha}\bar{b(\alpha)},
\end{align*}
so that \eqref{eq: zak as integral of product of matrices} becomes
\begin{align}
    \label{eq: final version zak, with def of P}
    \zak{j} = \frac{1}{2}\Im\left(\underbrace{-\int_{\sqrt{f}} \frac{1}{z} \dd z}_{=0}+(-1)^{j+1}\underbrace{a(v_2^2-v_1^2)\int_{\bar{b}} \frac{1}{z\sqrt{f(z)}} \dd z}_{\coloneqq P} - \int_{\bar{b}} \frac{1}{z} \dd z\right)+\BO(\delta).
\end{align}
Thus, we have shown the following theorem.

\begin{theorem}\label{thm: Zak non Hermitian}
    Consider a geometrical structure with $N=2$ and $\ell_1=\ell_2$ with a non-degenerate corresponding band structure. Then the Zak phase has the following asymptotic expansion:
    \begin{align}
    \zak{j} = (-1)^{j+1}\frac{s_1+s_2}{2s_1s_2}\Im\left((v_2^2-v_1^2)\int_{\gamma} \frac{1}{z\sqrt{f(z)}} \dd z\right)+\pi\mathds{1}_{\{x<s_1\}}(s_2)+\BO(\delta),
\end{align}
where $\gamma$ is the closed path
\begin{align*}
    \gamma(t)\coloneqq s_1\inv + s_2\inv e^{\i L t},
\end{align*}
and $f$ is defined along $\gamma$ as
\begin{align*}
    f(z)\coloneqq \left(s_1\inv+s_2\inv\right)^2(v_1^2-v_2^2)^2+4v_1^2v_2^2\abs{z}^2.
\end{align*}
\end{theorem}

One remarks already here that for the special case $v_1^2=v_2^2$ one obtains 
\begin{align}
    \zak{j} = \frac{1}{2}\Im\left(\int_{\gamma} \frac{1}{z} \dd z\right) = \begin{dcases}
        \pi & \mbox{if } \frac{1}{s_{2}}>\frac{1}{s_{1}}, \\
        0 & \mbox{if } \frac{1}{s_{2}}\leq\frac{1}{s_{1}},
        \end{dcases}
\end{align}
as in this case the Zak phase is known to be quantised \cite{ammari.davies.ea2020} so that we can drop the asymptotic factor. This proves \Cref{prop: Zak Hermitian}.

In general the integrals above are tedious to evaluate because of the non-holomorphicity of the integrand but we can check numerically that the integral is not zero and not constant, showing the non-quantisation of the Zak phase in the non-Hermitian case. Some values of this integral are shown in \Cref{tab:values of integral}.
\begin{table}[h]
    \centering
    \begin{tabular}{c|c|c|c|r}
        $s_{1}$ & $s_{2}$ & $v_1$ & $v_2$ & $\frac{1}{2}\Im(P)$ \\\hline\hline
        $1$ & $2$ & $1+1.38\i$ & $1-1.42\i$ & $0.408$\\\hline
        $1$ & $1$ & $1+1.38\i$ & $1-1.42\i$ & $2.420$\\\hline
        $1$ & $1$ & $1-1.42\i$ & $1+1.38\i$ & $-2.420$
    \end{tabular}
    \caption{Values of the perturbation factor $P$ {defined in \eqref{eq: final version zak, with def of P}} for various geometrical and material configurations.}
    \label{tab:values of integral}
\end{table}

In the PT-symmetric case, the system is degenerate. It has twice a double eigenvalue. The following result in that case can be shown explicitly. 
\begin{lemma}[PT-symmetric Zak phase]
    Assume that $N=2$ and $v_2=\bar{v_1}$. Then, 
    \begin{align*}
        \zak{j} = \BO(\delta).
    \end{align*}
\end{lemma}
\begin{proof}
    For this proof, we will denote by $\zak{j}(v)$ the Zak phase for $v_1=v$. Let also $\sigma = (1\ 2)$ be the permutation of two elements. Asymptotically, the Zak phase solely depends on the eigenvectors of the generalised capacitance matrix. We first show that $\zak{j}(v) = \zak{\sigma(j)}(\bar{v})$. To this end, we remark that using the definition of the capacitance matrix
    \begin{align*}
        \gencapmat= V^2 \capmat = 
        \begin{pmatrix}
            v^2 & 0 \\
            0 & \bar{v}^2
        \end{pmatrix}
        \begin{pmatrix}
            a & b(\alpha) \\
            \bar{b(\alpha)} & a
        \end{pmatrix}
    \end{align*}
    and the permutation matrix
    \begin{align*}
    P = \begin{pmatrix}
        0 & 1 \\
        1 & 0
    \end{pmatrix},
    \end{align*}
    we obtain the following relations:
    \begin{align*}
    \bar{V^2} = PV^2P\inv \quad\text{and}\quad \bar{\capmat} = PC^\alpha P = \mathcal{C}^{-\alpha}.
    \end{align*}
    So, 
    \begin{align*}
    \bar{\gencapmat} = P V^2 P\inv P \capmat P\inv = P \gencapmat
     P\inv, 
    \end{align*}
    and $\bar{\gencapmat}$ and $\gencapmat$ are similar via a permutation matrix. However, $\bar{\gencapmat} = \bar{V^2}\mathcal{C}^{-\alpha}$ and 
    so the eigenvectors of $\bar{V^2}\mathcal{C}^{-\alpha}$ are a permutation of the eigenvectors of $\gencapmat$. By symmetry around the origin of the Brillouin zone and \Cref{def: Zak phase Non Hermitian} of the Zak phase, we conclude that $\zak{j}(v) = \zak{\sigma(j)}(\bar{v})$.
    
    We now show that $\zak{j}(v) = -\zak{\sigma(j)}(\bar{v})$, which will complete the proof. Remark that the eigenvectors of the capacitance matrix given by \eqref{eq: U(alpha) gen cap mat} show that complex conjugating both material parameters leads to permuted and conjugated eigenvectors. \Cref{lemma: approx Zak formula} leads to the desired conclusion.
\end{proof}
\subsection{Localised edge modes generated by material-parameter defects}
\label{sec: non-Hermitian edge modes}
We have shown in \Cref{sec: Hermitian} that defects in the periodicity of the system can lead to edge modes. Recently, it has been shown that edge modes can also be generated in the non-Hermitian case via defects in the material parameters rather than in the geometry \cite{ammari.hiltunen2020}. We follow a similar approach to the one in  \cite{ammari.hiltunen2020},
showing that for a one dimensional chain of resonators one can explicitly identify the edge modes.

In this section, we consider the case of equally spaced dimers (i.e.,  with two identical resonators per cell), that is,
\begin{align}
\label{eq: geometrical assyumption edgemode non Hermitian}
    N=2,\quad \ell_1=\ell_2,\quad s_1=s_2.
\end{align}
We denote by $v_i^{(m)}$ the material parameter of the $i$-th resonator of the $m$-th dimer and similarly for the resonator itself.
\begin{figure}[h]
    \centering
    \begin{adjustbox}{width=\textwidth}
        \begin{tikzpicture}
        \draw[-,thick,dotted] (0.5,-1) -- (0.5,2);
        \draw[|-|,dashed] (0,1) -- (1,1);
        \node[above] at (0.5,1) {};
        \draw[|-|] (1,0) -- (2,0);
        \node[below] at (1.5,0) {$v_1^{(0)}$};
        \draw[-,dotted] (1,0) -- (1,1);
        
        \draw[|-|,dashed] (2,1) -- (3,1);
        \node[above] at (2.5,1) {};
        \draw[|-|] (3,0) -- (4,0);
        \node[below] at (3.5,0) {$v_2^{(0)}$};
        \draw[-,dotted] (2,0) -- (2,1);
        \draw[-,dotted] (3,0) -- (3,1);
        \draw[-,dotted] (4,0) -- (4,1);
        
        \begin{scope}[shift={(+4,0)}]
        \draw[|-|,dashed] (0,1) -- (1,1);
        \node[above] at (0.5,1) {};
        \draw[|-|] (1,0) -- (2,0);
        \node[below] at (1.5,0) {$v_1^{(1)}$};
        \draw[-,dotted] (1,0) -- (1,1);
        
        \draw[|-|,dashed] (2,1) -- (3,1);
        \node[above] at (2.5,1) {};
        \draw[|-|] (3,0) -- (4,0);
        \node[below] at (3.5,0) {$v_2^{(1)}$};
        \draw[-,dotted] (2,0) -- (2,1);
        \draw[-,dotted] (3,0) -- (3,1);
        \node at (4.5,.5) {\dots};
        \end{scope}
        
\begin{scope}[shift={(-4,0)}]

        \draw[|-|,dashed] (0,1) -- (1,1);
        \node[above] at (0.5,1) {};
        \draw[|-|] (1,0) -- (2,0);
        \node[below] at (1.5,0) {$v_2^{(-1)}$};
        \draw[-,dotted] (1,0) -- (1,1);
        
        \draw[|-|,dashed] (2,1) -- (3,1);
        \node[above] at (2.5,1) {};
        \draw[|-|] (3,0) -- (4,0);
        \node[below] at (3.5,0) {$v_1^{(-1)}$};
        \draw[-,dotted] (2,0) -- (2,1);
        \draw[-,dotted] (3,0) -- (3,1);
        \draw[-,dotted] (4,0) -- (4,1);
        \end{scope}
        
        \begin{scope}[shift={(-8,0)}]
        \node at (.5,.5) {\dots};
        \draw[|-|] (1,0) -- (2,0);
        \node[below] at (1.5,0) {$v_2^{(-2)}$};
        
        \draw[|-|,dashed] (2,1) -- (3,1);
        \node[above] at (2.5,1) {};
        \draw[|-|] (3,0) -- (4,0);
        \node[below] at (3.5,0) {$v_1^{(-2)}$};
        \draw[-,dotted] (2,0) -- (2,1);
        \draw[-,dotted] (3,0) -- (3,1);
        \draw[-,dotted] (4,0) -- (4,1);
        \end{scope}
        
        \end{tikzpicture}
    \end{adjustbox}
    \caption{Infinite structure with material parameter defect.}
    \label{fig: material paramter defect}
\end{figure}
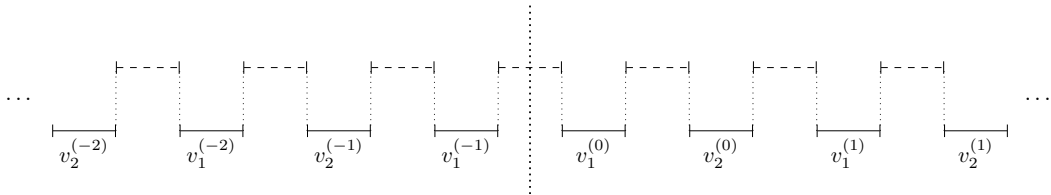
\begin{definition}[Localized edge mode]
    A solution $u$ to \eqref{eqn:DTN probelm} is said to be a simple eigenmode if it corresponds to a simple eigenvalue $\omega$ scaling as $\BO(\delta)$. A solution is said to be \emph{localised} if it is bounded in the $L^2$-sense, that is, $\int_{\R} \vert u(x)\vert^2 \dd x<\infty$.
\end{definition}

As we have seen in \Cref{prop: asymptotic of uj} and \Cref{lemma: approx of eigemodes via eigenvectors}, inside the resonators a subwavelength resonant mode is almost constant
\begin{align}
u(x) = u_i^m + \BO(\delta) \quad x\in D_i^{(m)} . \label{eq: almost constant inside resonator}
\end{align}

The following proposition is \cite[Proposition 4.2]{ammari.hiltunen2020}.
\begin{proposition}\label{prop: condition matrices localized solution edge mode}
Any localized solution $u$ to \eqref{eqn:DTN probelm} corresponding to a subwavelength frequency $\omega$ satisfies
\begin{align}
\frac{1}{\rho}
\capmat
\begin{pmatrix}
    \sum_{m\in\Z} u_1^m e^{\i \alpha m L}  \\
    \sum_{m\in\Z} u_2^m e^{\i \alpha m L} 
\end{pmatrix} = \omega^2
\begin{pmatrix}
    \sum_{m\in\Z} \frac{u_1^m e^{\i \alpha m L}}{(v_1^m)^2} \\
    \sum_{m\in\Z} \frac{u_2^m e^{\i \alpha m L}}{(v_2^m)^2}
\end{pmatrix}.
 \label{eq: localized solution must satisfy}
\end{align}
\end{proposition}

We consider the topological defect
\begin{align}
v_1^{(m)} = \begin{cases}
  v_1 & m\leq 0, \\
  v_2 & m > 0,
\end{cases}
\quad \mbox{and} \quad 
v_2^{(m)} = \begin{cases}
  v_2 & m\leq 0,\\
  v_1 & m > 0.
\end{cases}
\label{eq:topological defect}
\end{align}

The following lemma, which is \cite[Lemma 4.3]{ammari.hiltunen2020}, exploits the symmetry in the defect to obtain a decay rate of the mode.
\begin{lemma} \label{lem: exists a b for U_1 U_2 U_3 U_4}
Let 
\begin{align*}
U_1 = \sum_{m\leq0} u_1^m e^{\i \alpha m L}, \quad
U_2 = \sum_{m>0} u_1^m e^{\i \alpha m L}, \quad
U_3 = \sum_{m\leq 0} u_2^m e^{\i \alpha m L}, \quad
U_4 = \sum_{m> 0} u_2^m e^{\i \alpha m L}.
\end{align*}
Then, there exists some $b\in\C$ independent of $\alpha$ satisfying $\abs{b}< 1$ and
\begin{align*}
U_1 = b U_3,\quad U_4 = b U_2.
\end{align*}
\end{lemma}

Using the same notation as in \Cref{lem: exists a b for U_1 U_2 U_3 U_4}, we have
\begin{align*}
    \sum_{m\in\Z} u_1^m e^{\i \alpha m L} = U_2 + U_1 = U_2 + 
    b U_3, \quad \sum_{m\in\Z} u_2^m e^{\i \alpha m L} = U_3 + U_4 = U_3 + 
    b U_2.
\end{align*}

Furthermore, the topological defect \eqref{eq:topological defect} implies
\begin{align*}
    \sum_{m\in\Z} \frac{u_1^m e^{\i \alpha m L}}{\left(v_1^{(m)}\right)^2} = \frac{U_2}{v_2^2} + \frac{U_1}{v_1^2} = \frac{U_2}{v_2^2} + \frac{1}{v_1^2}b U_3, \quad \sum_{m\in\Z} \frac{u_2^m e^{\i \alpha m L}}{\left(v_2^{(m)}\right)^2} = \frac{U_3}{v_2^2} + \frac{U_4}{v_1^2} = \frac{U_3}{v_2^2} + \frac{1}{v_1^2 }b U_4.
\end{align*}

This allows us to rewrite \Cref{prop: condition matrices localized solution edge mode} as follows.

\begin{proposition}\label{prop: localized mode if eigenvalue independent of quasiperiodicity}
Assume that a structure as in \eqref{eq: geometrical assyumption edgemode non Hermitian} has a defect as in \eqref{eq:topological defect}. Then, there is a localised mode in the subwavelength regime corresponding to the frequency $\omega$ only if $B\inv \capmat A$ has an eigenvalue $\mu\in\C$ independent of the quasiperiodicity $\alpha$. Here,
\begin{align*}
A = \begin{pmatrix}
    1 & b \\
    b & 1
\end{pmatrix}, \quad 
B = \frac{1}{\delta}\begin{pmatrix}
    v_2^{-2} & bv_1^{-2} \\
    bv_1^{-2} & v_2^{-2}
\end{pmatrix}.
\end{align*}
\end{proposition}

Particular of the one-dimensional case is the explicit $\alpha$-dependence of the capacitance matrix, as seen in \Cref{def: Capacitance matrix}, which allows us to prove the next theorem.

Remark that it has been shown in \cite[Section 4.2]{ammari.hiltunen2020} that the decay rate of edge modes in the case of \eqref{eq: geometrical assyumption edgemode non Hermitian} must be either of
\begin{align}
b_{\pm} = \frac{1}{2} \left(3\left(1-\frac{v_1^2}{v_2^2}\right) \pm \sqrt{9\left(1-\frac{v_1^2}{v_2^2}\right)^2 + \frac{4 v_1^2}{v_2^2}}\right) \label{eq: b_pm},
\end{align}
whichever has magnitude smaller than $1$.
\begin{theorem}\label{thm: non-Hermitian edge mode}
  Assume that a one-dimensional structure as in \eqref{eq: geometrical assyumption edgemode non Hermitian} has a defect given by \eqref{eq:topological defect}. Then there always exists a simple eigenmode, which --- if $v_1 \neq \bar{v_2}$ or $v_1=\bar{v_2}\coloneqq v$ with $\sqrt{8}\abs{\Im(v^2)}\leq \abs{\Re(v^2)}$ --- is also localised.
  
  The frequency of the mode in the subwavelength regime satisfies
  \begin{align*}
  \omega = \pm\sqrt{\frac{\mu}{\ell_1}} + \BO(\delta),
  \end{align*}
  where 
  \begin{align*}
  \mu = \frac{\delta}{s_1}\frac{8v_1^2(-3v_ 1^2+v_2^2\sqrt{D}+3v_2^2)}{-7v_ 1^2+3v_2^2\sqrt{D}+9v_2^2}
  \end{align*}
  with $D\coloneqq \frac{9 v_{1}^{4}}{v_{2}^{4}} - \frac{14 v_{1}^{2}}{v_{2}^{2}} + 9$.
\end{theorem}
\begin{proof}
    The condition about localisation arises from the form of the decay rate given in \eqref{eq: b_pm}. For $v_1 \neq \bar{v_2}$ or $v_1=\bar{v_2}\coloneqq v$ with $\sqrt{8}\abs{\Im(v^2)}\leq \abs{\Re(v^2)}$ either $b_-$ or $b_+$ must have magnitude smaller than one. However, in the $v_1=\bar{v_2}\coloneqq v$ with $\sqrt{8}\abs{\Im(v^2)}> \abs{\Re(v^2)}$ case both $\vert b_\pm\vert =1$ making it impossible to have localised modes.

    For the eigenvalue computations, we assume without loss of generality that $s_{1}=1$ and introduce it again in the last step.

    The eigenvalues of $B\inv \capmat A$ are given by
    \begin{align*}
        &\mu_{j} = \underbrace{\frac{\delta v_1^2 v_2^2}{b^2v_2^4-v_1^4}}_{\coloneqq K}\left(\mathcal{C}_{11}^\alpha(b^2v_2^2-v_1^2)+b(v_2^2-v_1^2)\Re(\mathcal{C}_{12}^\alpha) \right.\\ &+ \left.(-1)^j \sqrt{(\mathcal{C}_{11}^\alpha(b^2v_2^2-v_1^2)+b(v_2^2-v_1^2)\Re(\mathcal{C}_{12}^\alpha))^2-(b^2-1)(b^2v_2^4-v_1^4)((\mathcal{C}_{11}^\alpha)^2-\abs{\mathcal{C}_{12}^\alpha}^2)}\right)
    \end{align*}
    and we will show that $\mu_0$ is independent of the quasiperiodicity. We assume without loss of generality that $b=b_-$, since the case $b=b_+$ can be proved similarly. Inserting the explicit coefficients of the capacitance matrix, we obtain
    \begin{align*}
        K&\left(\vphantom{\sqrt{v_{1}^{4} v_{2}^{4} \left(\left(2b^{2}-2\right) \left(b^{2} v_{2}^{4} - v_{1}^{4}\right) \left(\cos{\left(L \alpha \right)} - 1\right) + \left(2 b^{2} v_{2}^{2} + b \left(v_{1}^{2} - v_{2}^{2}\right) \left(\cos{\left(L \alpha \right)} + 1\right) - 2 v_{1}^{2}\right)^{2}\right)}}  \left(2 b^{2} v_{2}^{2} + b \left(v_{1}^{2} - v_{2}^{2}\right) \left(\cos{\left(L \alpha \right)} + 1\right) - 2 v_{1}^{2}\right) \right.
        \\&+ \left.\sqrt{v_{1}^{4} v_{2}^{4} \left(\left(2b^{2}-2\right) \left(b^{2} v_{2}^{4} - v_{1}^{4}\right) \left(\cos{\left(L \alpha \right)} - 1\right) + \left(2 b^{2} v_{2}^{2} + b \left(v_{1}^{2} - v_{2}^{2}\right) \left(\cos{\left(L \alpha \right)} + 1\right) - 2 v_{1}^{2}\right)^{2}\right)}\right),
    \end{align*}
    while in order to show independence from $\alpha$, it is enough to consider the term
    \begin{align}
        &b \left(v_{1}^{2} - v_{2}^{2}\right) \left(\cos{\left(L \alpha \right)} + 1\right) \nonumber
        \\&+ \sqrt{\left(2 b^{2} - 2\right) \left(b^{2} v_{2}^{4} - v_{1}^{4}\right) \left(\cos{\left(L \alpha \right)} - 1\right) + \left(2 b^{2} v_{2}^{2} + b \left(v_{1}^{2} - v_{2}^{2}\right) \left(\cos{\left(L \alpha \right)} + 1\right) - 2 v_{1}^{2}\right)^{2}}.
        \label{eq: alpha dep term eigenvalue edgemode}
    \end{align}
    Inserting into \eqref{eq: alpha dep term eigenvalue edgemode} the value of $b$ from \eqref{eq: b_pm}, we obtain after some careful algebraic manipulations
    \begin{align*}
        &\sqrt{2}\left(v_{1}^{2} - v_{2}^{2}\right) \sqrt{\underbrace{\left(9 v_{1}^{4} - 3 v_{1}^{2} v_{2}^{2} \sqrt{D} - 16 v_{1}^{2} v_{2}^{2} + 3 v_{2}^{4} \sqrt{D} + 9 v_{2}^{4}\right)}_{\coloneqq B}} (\cos{\left(L \alpha \right)} + 3) \\
        &- \left(v_{1}^{2} - v_{2}^{2}\right) \underbrace{\left(3 v_{1}^{2} - v_{2}^{2} \left(\sqrt{D} + 3\right)\right)}_{\coloneqq A} \left(\cos{\left(L \alpha \right)} + 1\right)
    \end{align*}
    with $D\coloneqq \frac{9 v_{1}^{4}}{v_{2}^{4}} - \frac{14 v_{1}^{2}}{v_{2}^{2}} + 9$. In order to verify independence from the quasiperiodicity, it is now enough to prove that $2 B = A ^2$. A direct computation shows that
    \begin{align*}
    A^2 &= Dv_2^4+6\sqrt{D}v_1^2v_2^2 - 6 \sqrt{D} v_2^4 + 9 v_1^4 - 18v_1^2v_2^2+9 v_2^4 \\
    &= 18 v_{1}^{4} - 6 v_{1}^{2} v_{2}^{2} \sqrt{D} - 32 v_{1}^{2} v_{2}^{2} + 6 v_{2}^{4} \sqrt{D} + 18 v_{2}^{4} \\
    &= 2 B.
    \end{align*}
    In particular, we have 
    \begin{align*}
    \mu_0 = \frac{\delta}{s_1}\frac{8v_1^2(-3v_ 1^2+v_2^2\sqrt{D}+3v_2^2)}{-7v_ 1^2+3v_2^2\sqrt{D}+9v_2^2}.
    \end{align*}
\end{proof}

\begin{figure}[p]
    \centering
     \begin{subfigure}[t]{0.48\textwidth}
        \centering
        \includegraphics[width=\textwidth,valign=t]{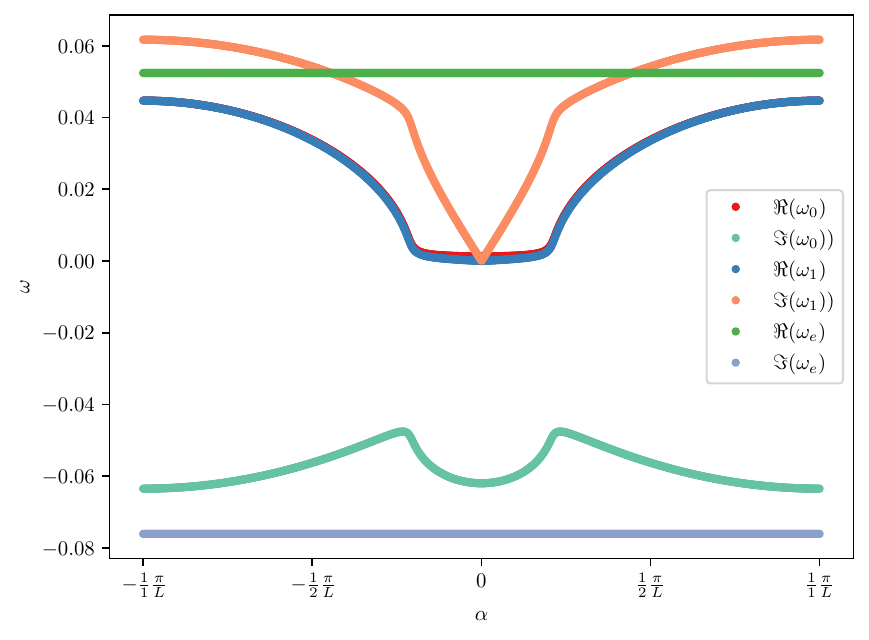}
        \caption{Representation depending on the quasiperiodicity.}
        \label{subfig: nonHermitian omegasalphas}
    \end{subfigure}
    \hfill
    \begin{subfigure}[t]{0.48\textwidth}
        \centering
        \includegraphics[width=\textwidth,valign=t]{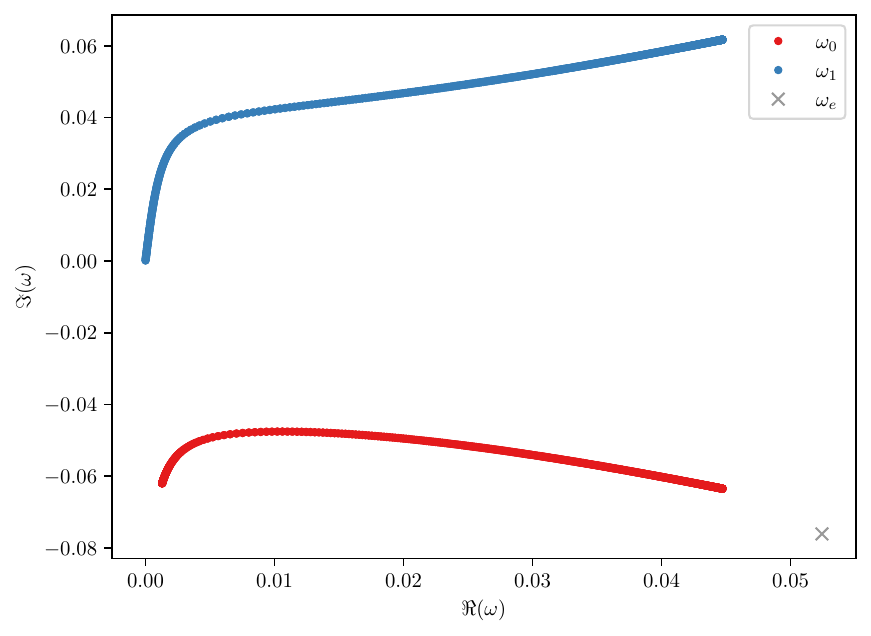}
        \caption{Traces on the complex plane.}
        \label{subfig: nonHermitian omegas in C}
    \end{subfigure}
        \\[1mm]
        \begin{subfigure}[t]{0.48\textwidth}
            \centering
            \includegraphics[width=\textwidth,valign=t]{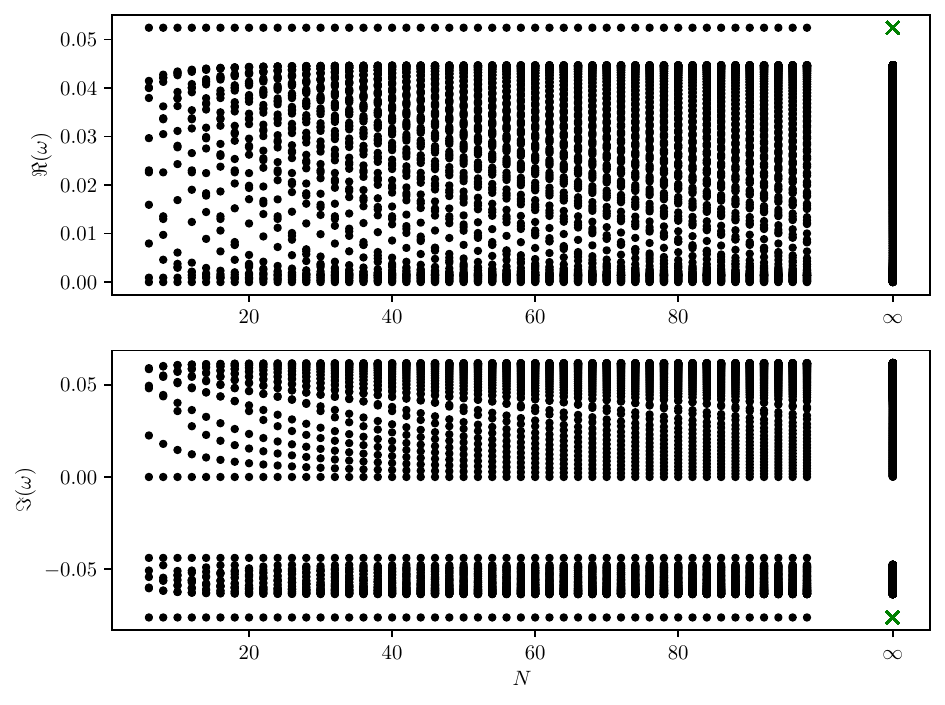}
            \caption{{Convergence of the subwavelength resonances with an increasing number of resonators. The solid line shows the spectrum of a periodic structure, while the green cross represents the predicted edge mode frequency.}}
            \label{subfig: convercence nonHermitian re and im}
    \end{subfigure}\hfill
    \begin{subfigure}[t]{0.48\textwidth}
        \centering
        \includegraphics[width=\textwidth,valign=t]{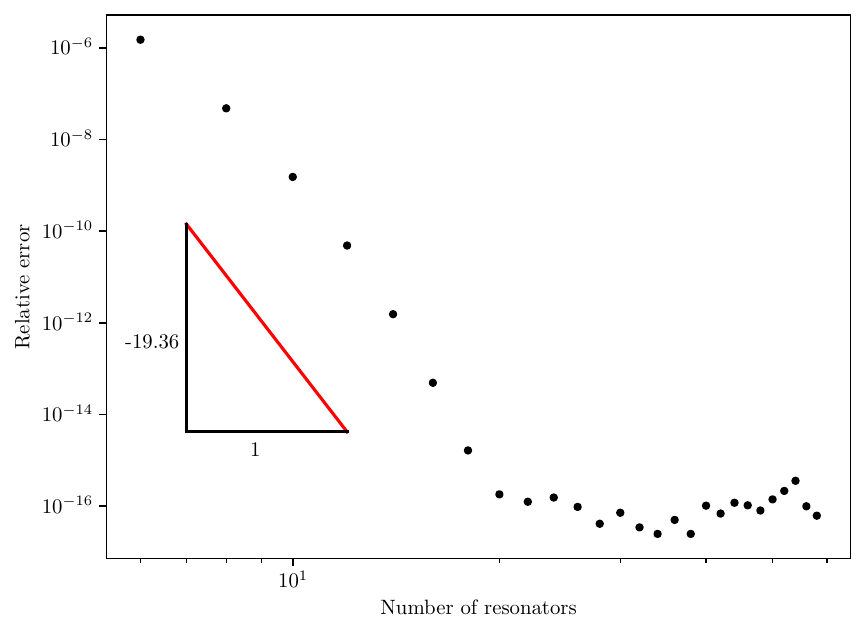}
        \caption{Convergence of the relative error $\abs{\omega_n-\omega}$ where $\omega_i$ is the edge subwavelength resonance in a finite structure with $i$ resonators and $\omega$ is the predicted subwavelength resonance from \Cref{thm: non-Hermitian edge mode}. The plot is in log-log scale.}
        \label{subfig: convercence realtive error nonHermitian}
\end{subfigure}
\\[1mm]
\begin{subfigure}[b]{0.48\textwidth}
    \centering
    \includegraphics[width=\textwidth]{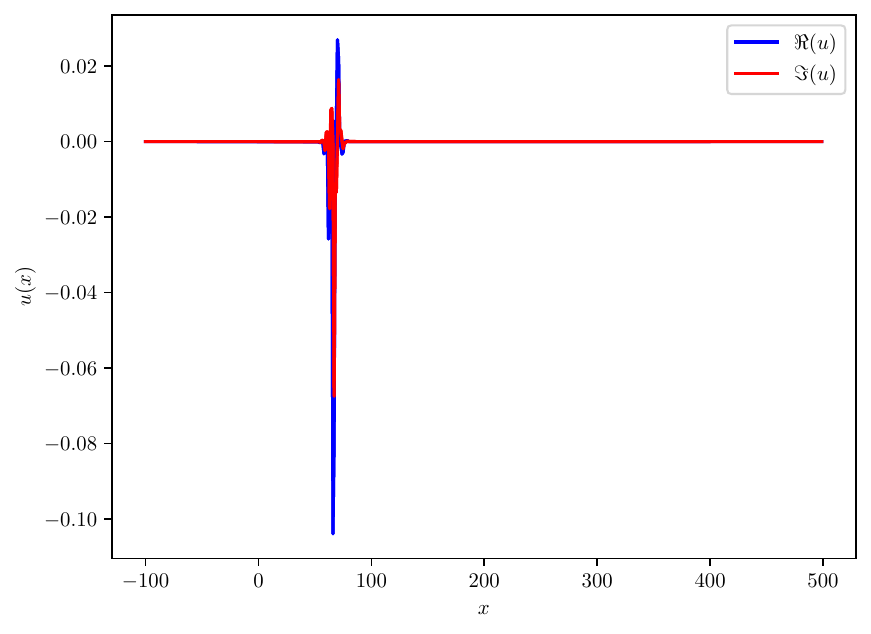}
    \caption{Localised edge mode for a finite but large array of $N=100$ resonators having a material parameter defect.}
    \label{subfig: plot non Hermitian edge mode}
\end{subfigure}
\hfill
\begin{subfigure}[b]{0.48\textwidth}
    \centering
         \includegraphics[width=\textwidth,valign=t]{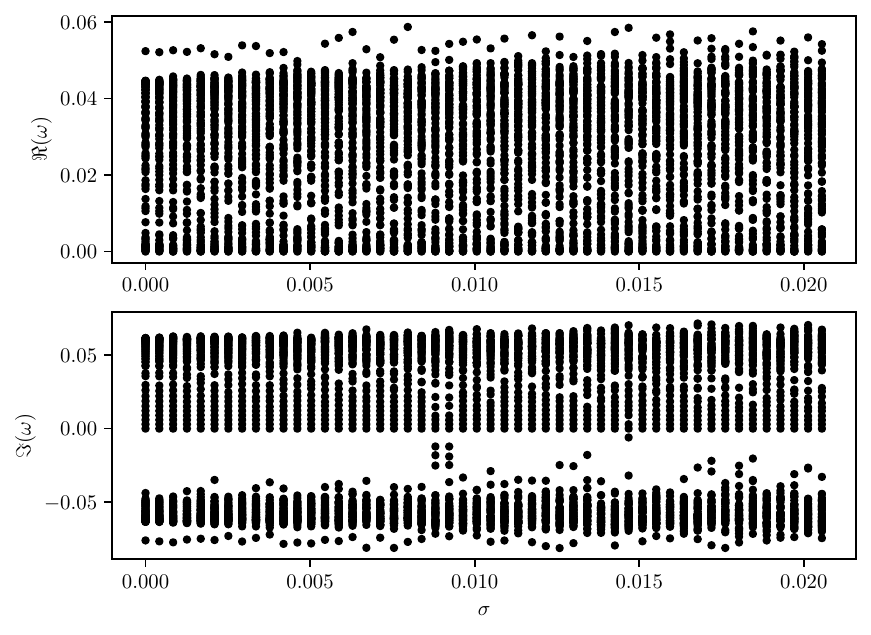}
         \caption{Stability of the edge mode with respect to perturbations in the material parameters.\\ \hfill }
         \label{subfig: stability wrt perturbations nonHermtian edgemode}
\end{subfigure}
    \caption{Edge modes generated by material parameter defects. For the infinite structure we used $N=2$, $\ell_i=1$, $s_i=1$, $v_1=1+1.38\i$, $v_2=1-1.42\i$.}
    \label{fig:ks conditon edge modes}
\end{figure}

As in \Cref{sec: Hermitian edge modes}, we provide some numerical simulations to visualise the edge mode. We first compute the bands $\omega^\alpha$ for the left infinite structure. These are shown in \Cref{subfig: nonHermitian omegasalphas} while \Cref{subfig: nonHermitian omegas in C} shows their traces in $\C$. In these plots, we add separately the edge mode frequency predicted by \Cref{thm: non-Hermitian edge mode}.

In \Cref{subfig: plot non Hermitian edge mode}, we show the edge mode computed for a finite but large array of resonators.

As \Cref{thm: non-Hermitian edge mode} provides an explicit formula for the edge mode frequency, it is particularly interesting to compare the subwavelength resonances of a finite structure with an increasing number of resonators with the band structure and predicted edge mode frequency --- both structures having the same material parameter defect. We do this in \Cref{subfig: convercence nonHermitian re and im}. In \Cref{subfig: convercence realtive error nonHermitian} we show that the convergence of the relative error $\abs{\omega_n - \omega}$ where $\omega_i$ is the subwavelength edge resonance in a finite structure with $i$ resonators and $\omega$ is the subwavelength edge resonance predicted by \Cref{thm: non-Hermitian edge mode} scales roughly as $\BO(n^{-19})$ and is in the order of magnitude of the machine precision with structures composed of 20 or more resonators.

As in the Hermitian case, we want to show that the edge mode is robust with respect to perturbations, this time in the material parameters. In \Cref{subfig: stability wrt perturbations nonHermtian edgemode}, we compute the subwavelength resonances of a finite but large array of $N=100$ resonators having a material parameter defect with some random perturbation given by
\begin{align*}
    \tilde{v_i} = v_i + (1\ \i)\cdot \epsilon_i, \quad \epsilon_i\sim \mathcal{N}(0,\Sigma),
\end{align*}
where
\begin{align*}
    \Sigma = \sigma\diag\left(\frac{1}{\sqrt{2}\abs{v_i}}\right).
\end{align*}
We remark that the edge mode predicted by \Cref{thm: non-Hermitian edge mode} is stable with respect to perturbations in the material parameters. The stability is, however, less strong than in the Hermitian case. This is to be expected due to the non-quantisation of the Zak phase in this setup. \Cref{subfig: stability wrt perturbations nonHermtian edgemode} shows that there is a second isolated frequency supported by this setup not predicted by \Cref{thm: non-Hermitian edge mode}. However, this frequency is not isolated from the bulk even for very small perturbations.

\section*{Acknowledgments}
This work was supported in part by the Swiss National Science Foundation grant number 200021-200307. The authors thank Erik Orvehed Hiltunen and Bryn Davies for insightful discussions.
\bigskip
\begin{center}
    \large
    \textbf{Code availability}
\end{center}
\medskip
The data that support the findings of this study are openly available at \newline \href{https://gitlab.math.ethz.ch/silvioba/edge-modes-1d}{https://gitlab.math.ethz.ch/silvioba/edge-modes-1d}.
\printbibliography

\end{document}
 